\documentclass{amsart}

\usepackage{amscd,amssymb,epsfig, epsf,pinlabel,graphicx}
\usepackage{epic,eepic}
\usepackage[dvipsnames]{xcolor}

\usepackage[all]{xy}
\SelectTips{cm}{}
\allowdisplaybreaks

\numberwithin{equation}{subsection}

\newtheorem{propo}{Proposition}[section]

\newtheorem{theor}[propo]{Theorem}
\newtheorem{lemma}[propo]{Lemma}

\theoremstyle{definition}

\theoremstyle{remark}

\newcommand{\egm}{\normalcolor{}}

\let\oldmarginpar\marginpar
\renewcommand\marginpar[1]{\oldmarginpar{\footnotesize #1}}

\newcommand{\ZZ}{\mathbb{Z}}

\newcommand{\Int}{\operatorname{Int}}

\newcommand{\card}{\operatorname{card}}

\begin{document}

\title[Loops in surfaces and star-fillings]{Loops in surfaces and star-fillings}

    \author[Vladimir Turaev]{Vladimir Turaev}

\begin{abstract} We discuss a new approach to computing the standard algebraic operations on homotopy classes of  loops in a surface:  the homological intersection number,  Goldman's Lie  bracket, and the author's Lie cobracket.  Our approach uses   fillings  of the surface by certain graphs.  
\end{abstract}

\maketitle

\section {Introduction}\label{Introduction}

The homological
 intersection number    of   loops    in an oriented surface~$\Gamma$   is computed  by deforming these loops  into a    transversal position and counting   their intersections  with signs determined by the orientation of~$\Gamma$. 
 This defines a skew-symmetric bilinear form   $\cdot_\Gamma: H_1(\Gamma) \times H_1(\Gamma) \to \ZZ$. 
  A   subtler way to     count    intersections  of loops  leads to   Goldman's Lie bracket 
   in the module    generated by   free homotopy classes of loops in~$\Gamma$, see  \cite{Go1}, \cite{Go2}. This bracket was  complemented  in \cite{Tu1}  by a  Lie cobracket  defined  in  terms of self-intersections of loops; for recent work on these operations  see \cite{AKKN1},
  \cite{AKKN2},   \cite{Hain}, \cite{Hain2}, \cite{Ka}, \cite{KK}, \cite{KK2}, \cite{LS}, \cite{Ma}. 
  For   compact~$\Gamma$,  the duality isomorphism   $H_1(\Gamma) \approx H^1(\Gamma, \partial \Gamma)$ carries   the  form~$\cdot_\Gamma$ into the composition of the cup-product    $ \cup : H^1(\Gamma, \partial \Gamma) \times  H^1(\Gamma, \partial \Gamma)\to  H^2(\Gamma, \partial \Gamma)$   with the linear map $ H^2(\Gamma, \partial \Gamma)\to \ZZ$   evaluating 2-cohomology classes  on the fundamental class of~$\Gamma$. So,  the  
 intersection number    of  two  elements of $H_1(\Gamma) $ can  be computed by taking the  dual cohomology classes, representing them by 1-cocycles on a triangulation of~$\Gamma$, evaluating  the  cup-product  of these 1-cocycles  on all  triangles of the triangulation, and  summing up the resulting values.  This method  allows us  to compute the algebraic number of intersections of  loops without ever considering the intersections themselves. The   aim of this  paper is to give   similar computations of the Lie bracket and cobracket mentioned above. To do it, we switch from  triangulations   to a more flexible language of     graphs  in surfaces.

In the rest of the introduction we focus on the case of surfaces with non-void boundary. Let 
~$\Gamma$ be  a compact connected oriented surface with $\partial \Gamma \neq \emptyset$. 
By a \emph{star} we   mean  an oriented  graph formed by   $n\geq 2$  vertices of degree~$1$ (the leaves),  a  vertex of degree~$n$ (the center), and~$n$ edges leading from the center to  the leaves.  The set of leaves of a star~$s$ is denoted by $\partial s$.   A \emph{star~$s$  in~$\Gamma$} is  a star embedded in~$\Gamma$ so that  $s\cap \partial \Gamma =\partial s$. The orientation  of~$\Gamma$  at the  center  of~$s$ determines a cyclic order  in the set ${\text {Edg}}(s)$ of   edges of~$s$. For   $e \in {\text {Edg}}(s)$   we let $e^+ \in {\text {Edg}}(s)$ be  the next   edge   with respect to this order.  We say that a loop   in~$\Gamma$ is \emph{$s$-generic} if it   misses   the vertices of~$s$,   meets all  edges of~$s$  transversely, and never  traverses a point of~$s$ more than once. It is clear that any loop in~$\Gamma$ can  be made $s$-generic by a small deformation.      Given an $s$-generic loop~$a$ in~$\Gamma$ and an edge   $e \in {\text {Edg}}(s)$ we let $a\cap e  $ be the set of points of~$e$ traversed by~$a$. For $p\in a\cap e$,   the intersection sign of~$a$ and~$e$ at~$p$ is denoted by $\mu_p (a) $.
The integer $a\cdot e =\sum_{p\in a \cap e}\mu_p(a)$ is the algebraic number of   intersections of~$a$ and~$e$. For   $s$-generic loops  $a,b$ in~$\Gamma$,  set
$$  a \cdot_s b = \sum_{e \in {\text {Edg}}(s)} \big (( a\cdot e) (b\cdot e^+) - (b\cdot e) ( a\cdot e^+) \big ). 
$$

\begin{theor}\label{s-inter}   The map $(a,b) \mapsto a\cdot_s b$ defines a skew-symmetric bilinear form $\cdot_s: H_1(\Gamma) \times H_1(\Gamma) \to  \ZZ$ depending only on the isotopy class of the star~$s$ in~$\Gamma$.
\end{theor}

The idea behind this  theorem  is to view each pair of points $p\in a\cap e, q\in b \cap e^+$  as a pseudo-intersection of~$a$ and~$b$, and to consider  the (skew-symmetrized)  algebraic number of such pairs. To recover the standard  homological intersecion form $\cdot_\Gamma$ in $H_1(\Gamma)$, we need one more  notion.  A \emph{star-filling} of~$\Gamma$  is a finite family~$F$ of  disjoint stars   in~$\Gamma$ such that each component of   the set $\Gamma \setminus \cup_{s\in F} s$    is a disk meeting  $\partial \Gamma$ at one or  two  open  segments. We explain in the body of the paper that~$\Gamma$  has    star-fillings. The following theorem computes the intersection form  $\cdot_\Gamma$ in terms of star-fillings.

\begin{theor}\label{s-inter2}   For any star-filling~$F$ of~$\Gamma$, we have
\begin{equation}\label{s-inter3}
 2\, \cdot_\Gamma=    \sum_{s\in F} \cdot_s: H_1(\Gamma) \times H_1(\Gamma) \to \ZZ. \end{equation}
\end{theor}

So, for any $x,y \in H_1(\Gamma)$, the integer  $\sum_{s\in F} x \cdot_s y$ is   even  and is equal to  $2\, x \cdot_\Gamma y$. As a consequence, the form $\cdot_\Gamma$ can be fully recovered from the forms $\{\cdot_s\}_{s\in F}$. Similar remarks apply to our computations of the bracket and cobracket below.

To state an analogue of Theorems~\ref{s-inter}  and~\ref{s-inter2}    for Goldman's Lie bracket,  we first recall the  definition of this bracket.  Let  ${\mathcal L}={\mathcal L}(\Gamma)$   be the set of free   homotopy classes of loops in~$\Gamma$  
 and let   $ M=M(\Gamma)$ be  the free abelian group with basis ${\mathcal L}$. For a loop~$a$ in~$\Gamma$, we let $\langle a \rangle\in \mathcal L \subset M$ be the free homotopy class of~$a$. For a  point $p \in \Gamma$ traversed by~$a$  once, we let $a_p$ be the loop starting  at~$p$ and going along~$a$ until the return to~$p$. We say that a pair of loops $a,b$ in~$\Gamma$ is \emph{generic} if these loops are transversal and do not meet at     self-intersections of~$a$ or~$b$. Then  the (finite) set of  intersection points of $a,b$ is denoted by   $a\cap b$. For  $p\in a \cap b$,  we let $\varepsilon_p (a,b)= \pm 1$ be the  intersection sign of~$a$ and~$b$ at~$p$. Also,  $a_p b_p$ stands for the  product of the loops $a_p, b_p$ based at~$p$. 
 Goldman's bracket  is the bilinear form  $[\, ,\, ]_\Gamma :M\times M\to M$ defined  on the basis $\mathcal L\subset M$ by 
 $$[ \langle a \rangle, \langle b \rangle ]_\Gamma=\sum_{p\in a\cap b} \varepsilon_p(a,b) \langle a_p b_p \rangle$$
 for any generic pair of  loops $a,b$ in~$\Gamma$. 
 The bracket $[\, ,\, ]_\Gamma$  is a well-defined  homotopy lift of the   form  $\cdot_\Gamma$ in $ H_1(\Gamma)$.

With any  star~$s$ in~$\Gamma$ we now associate a bilinear map $[\, ,\,]_s:M\times M\to M$.  For  any $s$-generic loops $a,b$ in~$\Gamma$ and any  points $p,q \in  s $ traversed respectively by $a,b$, we  pick a  path~$c$ from~$p$ to~$q$ in~$s$ and write $a_pb_q$ for the loop $a_p c b_qc^{-1}$. Clearly, the   free homotopy class of this loop  does not depend on the choice of~$c$.  Set 
$$  [ a, b]_s   = \sum_{e \in {\text {Edg}}(s)} \big ( \sum_{p\in a\cap e, q\in b\cap e^+} \mu_p(a)\,  \mu_q(b)  \langle  a_p b_q \rangle  -  \sum_{ p\in a\cap e^+, q\in b\cap e}
   \mu_p(a) \, \mu_q(b)  \langle      a_p b_q \rangle\big ) . 
$$ 
   
   \begin{theor}\label{s-inter++}   The map $(\langle  a \rangle  , \langle  b\rangle  ) \mapsto [ a, b]_s $ defines a skew-symmetric bilinear form $[\, ,\,]_s:M\times M\to M$ depending only on the isotopy class of~$s$ in~$\Gamma$.
  For any star-filling~$F$ of~$\Gamma$, we have
\begin{equation}\label{s-inter3++}
2\,   [\, ,\,]_\Gamma=  \sum_{s\in F} [\, ,\,] _s.  \end{equation}
\end{theor}

One may ask whether the bracket $[\, ,\,] _s$ shares the fundamental properties of Goldman's bracket, namely, whether it  satisfies the Jacobi identity and induces Poisson brackets  on the moduli spaces of~$\Gamma$.  In general, the answer to both questions is negative though some weaker results  hold true, see \cite{Tu2}.  Note also that in analogy with Goldman's bracket, Kawazumi and Kuno \cite{KK} defined an action of the Lie algebra $M$ on the modules generated by homotopy classes of paths in~$\Gamma$ with endpoints in $\partial \Gamma$. The Kawazumi-Kuno action can be computed similarly to Theorem~\ref{s-inter++} in terms of star-fillings. The same methods   work   for the double brackets of surfaces defined in  \cite{MT}, this will be discussed  elsewhere.

We next state  our results concerning  the Lie cobracket on loops in~$\Gamma$, see ~\cite{Tu1}.  For a loop~$a$ in~$\Gamma$,   set $\langle a \rangle_0 =\langle a \rangle  \in  \mathcal L \subset   M $  if~$a$ is non-contractible and $\langle a \rangle_0=0\in M $ if~$a$ is contractible. We say that the loop~$a$ is \emph{generic} if all its self-intersections are double transversal intersections. The set of self-intersections    of a generic loop~$a$   is finite and is denoted    $\# a  $.   The loop~$a$ crosses each point  $r\in \# a$ twice; we let $v^1_r, v^2_r$ be the tangent vectors of~$a$ at~$r$ numerated so    that the pair $(v^1_r, v^2_r)$ is positively oriented. For $i=1,2$,   let $a^i_r$ be the loop   starting in~$r$ and going  along~$a$   in the direction of the vector $v^i_r$    until the first return to~$r$. Up to parametrization, $a=a^1_ra^2_r$ is the product of the  loops $a^1_r, a^2_r$. The cobracket $\nu_\Gamma$ is a linear map $ M  \to M  \otimes M $ defined on the basis ${\mathcal L}  \subset M$ by
$$\nu_\Gamma( \langle a \rangle )=\sum_{r\in \# a} \, \langle a^1_r \rangle_0 \otimes \langle a^2_r \rangle_0  - \langle a^2_r \rangle_0 \otimes \langle a^1_r \rangle_0  $$
for any generic loop~$a$ in~$\Gamma$. The cobracket $\nu_\Gamma$ is skew-symmetric in the sense that its composition with the permutation in $M \otimes M$ is equal to $-\nu_\Gamma$.

For any  star~$s$ in~$\Gamma$ we define a  linear map $\nu_s:M \to M  \otimes M$.  Given an $s$-generic loop~$a$ in~$\Gamma$  and   points $p_1, p_2$ of~$ s $ traversed by~$a$, we   write $a_{p_1, p_2}$ for the loop going from~$p_1$ to~$p_2$  along~$a$ and then going back to~$p_1$
along a path in~$s$.   The   free homotopy class of this loop  does not depend on the choice of the latter path.  Set 
$$  \nu_s(\langle a \rangle) $$
$$= \sum_{e \in {\text {Edg}}(s)} \,  \sum_{p_1\in a\cap e, p_2\in a\cap e^+} \mu_{p_1}(a)  \,\mu_{p_2 }(a) \, \big ( \langle a_{p_1, p_2} \rangle_0 \otimes \langle a_{p_2,p_1} \rangle_0  - \langle a_{p_2, p_1} \rangle_0  \otimes \langle a_{p_1, p_2} \rangle_0   \big ) .
$$

  \begin{theor}\label{s-inter++cobr}   The map $\langle  a \rangle   \mapsto \nu_s(\langle  a \rangle ) $ defines a skew-symmetric cobracket $M  \to M  \otimes M $  depending only on the isotopy class of~$s$ in~$\Gamma$.
  For any star-filling~$F$ of~$\Gamma$, we have
\begin{equation}\label{s-inter3++cobr}
  2 \,  \nu_\Gamma=    \sum_{s\in F} \nu_s.  \end{equation}
\end{theor}

The cobracket $\nu_\Gamma$ has a refinement depending on a framing of~$\Gamma$, see \cite{Tu1}, Section~18.1 and \cite{AKKN2}; it would be interesting to extend Theorem~\ref{s-inter++cobr}   to this refined cobracket.

 Theorems \ref{s-inter}--\ref{s-inter++cobr} are proved in Section~\ref{Proof of Theorems and their extension to closed surfaces} where we also discuss the case of closed surfaces.  The proofs of Theorems \ref{s-inter}--\ref{s-inter++cobr} are based on 
  the theory  of   quasi-surfaces developed in Sections \ref{A topological example}--\ref{Stars}.

This work was  supported by the NSF grant DMS-1664358.

\section{Quasi-surfaces}\label{A topological example}

\subsection{Basics}\label{Terminology}\label{notat} By  a   \emph{surface} we mean  a  smooth 2-dimensional manifold  with   boundary.  
A \emph{quasi-surface} is a  topological space~$X$ obtained by gluing an oriented   surface~$\Sigma$ to a topological space~$Y$ along a continuous map   $   \alpha  \to Y$ where $ \alpha \subset \partial \Sigma $ is the  union of a finite number of    disjoint  segments  in $\partial \Sigma$.  
Clearly,    $Y\subset X$ and   $X\setminus Y =\Sigma\setminus \alpha$.  We call~$Y$ the \emph{singular core} of~$X$ and call~$\Sigma$ the \emph{surface core} of~$X$. 
We   fix a closed    neighborhood  of~$\alpha$ in~$\Sigma$ and 
  identify   it  with  $\alpha \times [-1,1]$ so that 
  $$\alpha=\alpha \times \{-1\} \quad  \quad {\text {and}} \quad \quad  \partial \Sigma  \cap (\alpha \times [-1,1])=\alpha   \cup (\partial \alpha \times [-1,1]).$$
  The  surface
$$\Sigma'=\Sigma \setminus  (\alpha \times [-1,0)) \subset \Sigma \setminus \alpha \subset X   $$  is  a  copy of $\Sigma$    embedded in~$X$. 
We    provide~$\Sigma'$  with the   orientation induced  from that of~$\Sigma$ and   call~$\Sigma'$  the \emph{reduced surface core} of~$X$.

Set  $\pi_0=\pi_0(\alpha)=\pi_0(\alpha \times \{0\})$.   For    $k\in  \pi_0 $,   we   let   $ \alpha_k $ be  the corresponding   segment component of $\alpha \times \{0\} \subset \partial \Sigma' \subset X$. 
  We call   $\alpha_k$    the  \emph{$k$-th gate} of~$X$.
The gates $\{\alpha_k\}_{k\in \pi_0}$      separate $
\Sigma'\subset X$ from the rest of~$X$.        A   \emph{gate orientation}  of~$X$     is   an orientation of all gates. 
    Gate orientations of~$X$  canonically correspond to orientations of the 1-manifold~$\alpha $.
  For a gate  orientation~$\omega$ of~$X$, we let~$\overline \omega$   be  the  gate orientation of~$X$   opposite to~$\omega$ on all  gates. 
 
   We  keep the notation  $X,Y, \alpha, \Sigma, \Sigma', \{\alpha_k\}_k, \pi_0$      till  the end of Section~\ref{Stars}.

\subsection{Loops in~$X$}\label{koko}\label{kokoff}\label{koko89} 
By a \emph{loop} in~$X$ we mean a   continuous map  $a:S^1\to X$. 
 A \emph{generic loop}~$a$ in~$X$ is  a   loop   in~$X$ such that (i)  all  branches of~$a$   in~$\Sigma'$ are smooth immersions meeting $\partial \Sigma'$ transversely   at a finite set of points lying  in the  interior  of the gates, and (ii)  all self-intersections of~$a$ in $\Sigma'$ are double transversal intersections  in $\Int(\Sigma')=\Sigma' \setminus \partial \Sigma'$.  The set of self-intersections in $\Sigma'$    of a generic loop~$a$   is denoted by   $\# a $. This set is finite and  lies in $\Int(\Sigma')$. Using cylinder neighborhoods of the gates, it is easy to see that   any loop in~$X$ may be    transformed  into a generic loop by a small deformation.

More generally, a finite family of   loops in~$X$    is  \emph{generic}  if    these loops  are generic and all  their intersections   in~$\Sigma'$ are double transversal intersections   in $\Int(\Sigma')$. In particular,   these loops  can  not meet at the gates. 
As above, any finite family of loops   in~$X$ may be transformed into a generic  family  by a small deformation.

 For a loop $a$ in~$X$ and any point~$p$ of $a(S^1)$ which is not a self-intersection of~$a$ (i.e., which is traversed by~$a$ only once),
 we let $a_p$ be the   loop  which starts at~$p$ and goes along~$a$ until coming back to~$p$. For any  $k\in \pi_0$, we set $a\cap \alpha_k=a(S^1) \cap \alpha_k $.   If the loop~$a$ is generic  then it never traverses a point of    $a \cap \alpha_k$   more than once and the set $a\cap \alpha_k$ is finite.

\subsection{Local moves}\label{moves}   
We define six  local  moves $L_0 - L_5$  on  a  generic loop~$a$  in~$X$  keeping its  free homotopy class. The move $L_0$ is a deformation of~$a$  in the class of generic loops. This move preserves the number $\card(\# a)$.
  The   moves $L_1 -  L_3$   modify~$a$  in a small disk in $\Int(\Sigma')$.   The   move $L_1$ adds a small curl to~$a$ and increases $\card(\# a)$  by~$1$. The  move $L_2$ pushes a branch of~$a$ across another branch of~$a$ increasing $\card(\# a)$  by~$2$.  The  move $L_3$ pushes a branch of~$a$  across a double point  of~$a$ keeping $\card(\# a)$. The   moves $L_4, L_5$   modify~$a$  in a   neighborhood of  a gate $\alpha_k$ in~$X$.     The  move $L_4$ pushes a branch of~$a$ across    $\alpha_k$ increasing $\card(a \cap \alpha_k)$ by~2 and    keeping  $\card(\# a)$. The  move  $ L_5$  pushes a double point of~$a$ across $\alpha_k$  keeping $\card(a \cap \alpha_k)$ and  decreasing   $\card(\# a)$  by  $ 1$.  
We call the moves $L_0 - L_5$ and their inverses   \emph{loop moves}.  It is clear  that   generic loops in~$X$ are freely homotopic  if and only if  they can be related  by a finite sequence of loop moves.

\section{Homological  intersection  forms}\label{section2}

We define   homological   intersection forms in $H_1(X)$. Here and below,  by $H_1(X)$ we mean the 1-homology of the underlying topological space of the quasi-surface~$X$.

\subsection {The intersection  form of $(X, \omega)$}\label{fofo}  
Given a gate orientation~$\omega$ of~$X$, we   define   a   bilinear form in $ H_1(X)$  
  called the \emph{homological intersection  form of the pair $(X, \omega)$}.   The idea  is to properly  position the loops   in~$X$  near the gates and then  to count their algebraic number of intersections  in~$\Sigma' $. We begin with definitions. 

For any      points $p,q$ of a gate $ \alpha_k$,  we  say that~$p$ \emph{lies on  the $\omega$-left of~$q$} and write  $p<_\omega q$
if $p\neq q$ and  the $\omega$-orientation  of $\alpha_k$   leads from~$p$ to~$q$.  
We say that an (ordered)  pair  of   loops $a,b$   in~$X$  is     \emph{$\omega$-admissible}  if it is generic  (in the sense of Section~\ref{koko}) and   the crossings of~$a$ with any gate  lie on the $\omega$-left of the  crossings of~$b$ with this gate.   Taking  a   generic pair of  loops $a,b$  in~$X$ and pushing the branches of~$a$  crossing the gates to the $\omega$-left   of the crossings of~$b$ with the gates, we obtain an  $\omega$-admissible pair of loops.
Thus,  any pair of loops in~$X$ may be deformed into an $\omega$-admissible pair. 

For each  generic pair of loops $a,b$ in~$X$, we    consider the finite set    $$a\cap b= a(S^1) \cap b(S^1)\cap \Sigma' \subset   \Int(\Sigma') .$$   For     $r \in a\cap b$,   set $\varepsilon_r  (a,b) =  1$ if  the tangent vectors of~$a$ and~$b$ at~$r$ form a    positive basis in the tangent space of~$\Sigma'$ at~$r$ and   set $\varepsilon_r   (a,b)=-1$ otherwise.

\begin{lemma}\label{1alele}   For any $\omega$-admissible pair $a,b$ of   loops in~$X$,   the integer  \begin{equation}\label{actionqq} a  \bullet_{X, \omega} b=  \sum_{r\in a\cap b} \varepsilon_r  (a,b)   \end{equation} depends only on  the homology classes of $a,b$ in  $ H_1(X)$. The formula $(a,b) \mapsto   a  \bullet_{X, \omega} b$ defines a bilinear form $$
\bullet_{X,\omega}: H_1(X)\times H_1(X) \to \ZZ
$$
\end{lemma}
\begin{proof} For each $k\in \pi_0$,    one  endpoint  of  the gate $  \alpha_k$ lies on the $\omega$-left of the other  endpoint. Pick disjoint closed  segments  $\alpha_k^- \subset \alpha_k $ and $\alpha_k^+\subset \alpha_k $     containing  these  two endpoints respectively.  Clearly,   $p<_\omega q$ for all $p\in \alpha_k^- $ and $q\in \alpha_k^+$. 
We  say that a loop  in~$X$ is \emph{$\omega$-left} (respectively, \emph{$\omega$-right}) if it is generic and meets the gates of~$X$ only  at points of $\cup_k \alpha^-_k$ (respectively, of $\cup_k \alpha^+_k$). Given an $\omega$-admissible pair of loops $a,b$ in~$X$, we can  push the branches of~$a$ crossing the gates to the left and  push the branches of~$b$ crossing the gates to the right without creating or destroying   intersections between~$a$ and~$b$. Consequently,~$a$ is homotopic (in fact, isotopic) to an $\omega$-left loop $a'$ and~$b$ is homotopic to an $\omega$-right loop $b'$ such that  $ a  \bullet_{X, \omega} b= a'  \bullet_{X,\omega} b'$. Since $\alpha^-_k $ is a  deformation retract  of   $\alpha_k  $ for all~$k $, any $\omega$-left loops     homotopic    in~$X$ are homotopic in the class of $\omega$-left loops. Similarly, any   $\omega$-right loops     homotopic   in~$X$ are homotopic in the class of $\omega$-right loops. Such homotopies  of the loops  $a', b'$ expand as compositions of loop moves    keeping  $a', b'$  respectively  $\omega$-left and $\omega$-right. The latter moves  obviously preserve   $a'  \bullet_{X,\omega} b'$. Therefore  the  integer $ a  \bullet_{X, \omega} b=a'  \bullet_{X,\omega} b'$ depends only on the (free)  homotopy classes of  $a,b$ in~$X$. Moreover, since  $ a  \bullet_{X, \omega} b$ depends  linearly  on~$a$ and~$b$,  it depends only on the  homology classes of $a,b$. This implies the  claim of the lemma.
\end{proof}

We stress  that the crossings  in  $  X\setminus \Sigma'$ of  an $\omega$-admissible pair of  loops $a,b$     do not contribute to $a  \bullet_{X, \omega} b$. Note also that for such $a,b$,  the   pair $b,a$ is $\overline\omega$-admissible. 
 Using these pairs to compute   $a  \bullet_{X,\omega} b$ and $b  \bullet_{X, \overline \omega} a$, we obtain    the same  terms with opposite signs. 
Hence, for any $x,y \in H_1(X)$,   \begin{equation}\label{chorisdf}     x  \bullet_{X, \omega} y   = - 
y  \bullet_{X, \overline \omega} x    .\end{equation}
  
  \subsection{The intersection  form of $X$}\label{fofofor}  We state the main result
   of this section.
  
   \begin{theor}\label{1aee}  The skew-symmetric bilinear form $\bullet_{X}: H_1(X) \times H_1(X) \to \ZZ$  defined by
$$
x \bullet_{X} y  =   
 x  \bullet_{X, \omega} y   - y  \bullet_{X,\omega}  x $$
 for all $x,y \in H_1(X)$ and a gate orientation~$\omega$ of~$X$  does not depend on~$\omega$.
 \end{theor}

We will prove this theorem in Section~\ref{mhllhodiff}.
 We call   $\bullet_{X}$ the \emph{homological intersection  form of~$X$}. Both   $ \bullet_{X,\omega}$ and $\bullet_{X}$   generalize  the   intersection form in the homology   of~$\Sigma'$:  the value of $ \bullet_{X,\omega}$    (respectively,  of $\bullet_{X}$)  on any pair of  homology classes  of loops in $\Sigma'\subset X$    is equal to    the usual   intersection number  of these loops in~$\Sigma'$ (respectively,  twice this number).  The image of the inclusion homomorphism $H_1(Y)\to H_1(X)$ annihilates both $ \bullet_{X,\omega}$ and $\bullet_{X}$.
 
 To prove Theorem~\ref{1aee},   we  study
  the dependence of  $
\bullet_{X,\omega}$ on~$\omega$. We start with notation. For a  generic loop~$a$ in~$X$, the \emph{sign} $\varepsilon_p(a)$ of~$a$ at a  point $p\in a\cap \alpha_k$ is    $+1$ if~$a$ goes near~$p$ from $X\setminus \Sigma'$ to $\Int (\Sigma')$ and $ -1$ otherwise. The  linear map $v_k: H_1(X)\to \ZZ$ \lq\lq dual'' to the gate $\alpha_k$  carries the homology class of any generic loop~$a$ to $\sum_{p\in a \cap \alpha_k} \varepsilon_p(a)$.    For any   loops $a,b$ in~$X$,  we define  the  set of triples   
$$T(a,b)=\{(k,p,q) \, \vert \, k\in \pi_0, p\in a\cap \alpha_k, q\in b\cap \alpha_k, p\neq q \}.$$
Given a gate orientation~$\omega$ of~$X$, we     set  
$$T_\omega (a,b)=\{(k,p,q) \in T(a,b) \, \vert \, q<_\omega p  \} \subset T(a,b).$$ 
For $k\in \pi_0$, we   set  $\varepsilon(\omega, k)=+1$ if the $\omega$-orientation of~$\alpha_k$ is compatible with the orientation of~$\Sigma'$, i.e.,  if the pair (a $\omega$-positive tangent vector  of  $\alpha_k \subset \partial \Sigma'$, a   vector  directed inside~$\Sigma'$) is  positively oriented in~$\Sigma'$. Otherwise,   $\varepsilon(\omega, k)=-1$.


\begin{lemma}\label{1aeee}    For any gate orientation~$\omega$ and any   $x,y \in H_1(X)$ represented by a generic pair of loops $a,b$ in~$X$, we have 
\begin{equation}\label{chorieemvmmve}
 x  \bullet_{X, \omega} y    =  \sum_{r\in a\cap b } \varepsilon_r   (a,b) +\sum_{(k,p,q) \in T_{\omega} (a,b)}\,    \varepsilon(\omega, k)   \,\varepsilon_p(a)  \,\varepsilon_q(b). \end{equation}
\end{lemma}
\begin{proof}  Consider an $\omega$-admissible pair of  loops $a',b$  where~$a'$ is obtained from~$a$ by pushing its branches   crossing the gates to the $\omega$-left of the branches of~$b$ crossing  the   gates.  This  transformation modifies~$a$ in a small neighborhood of the gates so that   $a', b$ have the same intersections in $\Sigma'$ as $a,b$ plus one  additional intersection  $r=r(k,p,q) \in \Sigma'$   for each  triple $(k,p,q)\in T_\omega (a,b)$. It is easy to check  that  $\varepsilon_r  (a', b)= \varepsilon(\omega, k) \, \varepsilon_p(a)\,  \varepsilon_q(b)$.  Consequently,    $$
 x  \bullet_{X, \omega} y =
a'  \bullet_{X,\omega} b=\sum_{r\in a'\cap b'} \varepsilon_r  (a',b) $$ 
$$=\sum_{r\in a\cap b} \varepsilon_r  (a,b)+\sum_{(k, p,q) \in T_{\omega} (a,b)} \varepsilon(\omega, k) \, \varepsilon_p(a)\,  \varepsilon_q(b).$$
\end{proof}

 Formula \eqref{chorieemvmmve} generalizes   \eqref{actionqq} because $T_\omega (a,b)=\emptyset$   for any  $\omega$-admissible pair  of loops $a,b$.


 \subsection{Proof of Theorem~\ref{1aee}}\label{mhllhodiff} For any ${l}\in \pi_0$,  we let ${l}\omega$ be the gate orientation   obtained from~$\omega$ by inverting the direction of the gate~$ \alpha_{{l}}$ while  keeping  the directions of all  other gates.
We claim that for any $x,y\in H_1(X)$,
\begin{equation}\label{chori}
x  \bullet_{X, {l} \omega} y =   
 x  \bullet_{X, \omega} y   - \varepsilon(\omega, {l}) \,  v_{{l}}(x) \,  v_{{l}}(y). \end{equation}
Indeed, pick an $\omega$-admissible  pair of loops $a,b$ representing respectively $x,y$. We  compute $x \bullet_{X,\omega} y  =a \bullet_{X,\omega} b $ from the definition and compute $x \bullet_{X, {l}\omega} y  = a \bullet_{X,{l} \omega} b $ from  Lemma~\ref{1aeee}.   The resulting   expressions   differ   in the sum associated with    $ T_{{l}\omega} (a,b)$. Since the pair $a,b$ is $\omega$-admissible, the set $T_{{l}\omega} (a,b) $  consists of  all triples $({l}, p, q)$ with  $ p\in a \cap \alpha_{{l}} ,q\in b \cap \alpha_{{l}}$. Therefore
$$x  \bullet_{X, {l} \omega} y    =
 x  \bullet_{X, \omega} y + \sum_{p\in a \cap \alpha_{{l}} ,q\in b \cap \alpha_{{l}}}\, \varepsilon({l} \omega, {l})\,  \varepsilon_p(a) \,  \varepsilon_q(b)$$
 $$= 
 x  \bullet_{X, \omega} y  - \varepsilon(\omega, {l}) \,  v_{{l}}(x) \,  v_{{l}}(y).$$
 Formula~\eqref{chori}  implies that
$$x \bullet_{X, {l}\omega} y - y\bullet_{X, {l}\omega} x =  x  \bullet_{X, \omega} y   - y  \bullet_{X,\omega} x $$
for all ${l}\in \pi_0$. This implies the claim of the theorem.

%

\subsection{Computation of $\bullet_X$}\label{The edbdbdform bullet}   To compute $x \bullet_X y$ for  $x,y\in H_1(X)$ we will use the following method. Pick a generic pair of loops $a,b$ in~$X$ representing $x,y$. Let   $\omega_0$ be the  orientation of the gates induced by  the orientation of $\Sigma'\subset \Sigma$ so that $\varepsilon(\omega_0, k) =1$ for all $k\in \pi_0$.  Lemma~\ref{1aeee}  implies that 
  $$
x \bullet_X y = x  \bullet_{X, \omega_0} y   - y  \bullet_{X,\omega_0} x   $$
$$ = 2 \sum_{r\in a\cap b} \varepsilon_r  (a,b)   +\sum_{(k,p,q) \in T(a,b) }\,   \delta  (p,q)    \,\varepsilon_p(a)  \,\varepsilon_q(b)  $$
where $ \delta  (p,q)=1$ for    $q<_{\omega_0} p$ and 
$ \delta  (p,q)=-1$ for  $p <_{\omega_0} q$.     If $a\cap b =\emptyset$, then
\begin{equation}\label{chorieee+bbb+veryryrsss}
x \bullet_X y =  \sum_{(k,p,q) \in T(a,b) }\,   \delta  (p,q)    \,\varepsilon_p(a)  \,\varepsilon_q(b)  . \end{equation}

 \section{The intersection brackets}\label{section3}

We define homotopy    intersection brackets   refining the homological   forms above.

\subsection{The  brackets}\label{The form bullet}   
Let   ${\mathcal L} = {\mathcal L} (X)$   be the set of free   homotopy classes of loops in the quasi-surface~$X$  
 and   let  $ M=M(X)$ be   the free abelian group with basis ${\mathcal L}$. 
Pick a gate orientation~$\omega$ of~$X$. 
By Section~\ref{fofo}, any  pair    $x,y \in {\mathcal L} $ can be represented by an  $\omega$-admissible pair of loops $a,b$ in~$X$. For  a point $r\in a\cap b$,  consider the loops  $a_r, b_r$  which are reparametrizations of $a, b$     based at~$r$.   Consider the product loop $a_r b_r$   and  set
\begin{equation}\label{deff} [x,y]_{X,\omega}=\sum_{r\in a\cap b} \varepsilon_r   (a,b) \langle a_r b_r \rangle \in  M \end{equation} where for a    loop~$c$ in~$X$, we let  $\langle c \rangle \in {\mathcal L} \subset M$ be its free homotopy class.  The sum on the right-hand side of \eqref{deff} is an algebraic sum of all possible ways to graft~$a$ and~$b$ at their intersections in~$\Sigma'$. It is straightforward to see that this sum
  is preserved under   all loop moves on $a,b$ keeping  this pair   $\omega$-admissible. Hence,  $[x,y]_{X,\omega}$ does not depend on the choice of $a,b$ in the   homotopy classes  $x,y$. Extending the map  $(x,y) \mapsto [x,y]_{X,\omega}$ by bilinearity, we obtain a  bilinear bracket $[-, -]_{X,\omega}$ in~$M$.
The proof of Formula~\eqref{chorisdf} applies here and  shows that  for any $x,y\in M$, 
\begin{equation}\label{chorisdfbulle}   
 [x,y]_{X,\omega} =- [y,x]_{X, \overline \omega}.   \end{equation}
 
We now compute the bracket  $[x,y]_{X,\omega}$ from an arbitrary   generic pair of    loops $a,b$ representing $x,y$. Note that for  any points $p  \in a\cap \alpha_k$, $q\in b\cap \alpha_k$ on the same gate,  we can   multiply the loops $a_p, b_q$  based at $p,q$ using a path   connecting  $p,q$ in $\alpha_k$. The product  loop determines an  element of $\mathcal L$   denoted $\langle a_p b_q \rangle$. 

\begin{lemma}\label{1aeee++}   Let $x,y \in {\mathcal L}$  be represented by a generic pair of loops $a,b$. Then 
\begin{equation}\label{chorieee++}
[x,y]_{X,\omega} =  \sum_{r\in a\cap b} \varepsilon_r  (a,b) \langle a_r b_r \rangle +\sum_{(k,p,q)\in T_\omega (a,b) }\,    \varepsilon(\omega, k)   \,\varepsilon_p(a)  \,\varepsilon_q(b) \langle a_pb_q \rangle. \end{equation}
\end{lemma}

The proof repeats the proof of Lemma~\ref{1aeee} with   obvious modifications.  If $a\cap b=\emptyset$, then  \eqref{chorieee++} simplifies to   
\begin{equation}\label{chorieee++vj}
[x,y]_{X,\omega}   =  \sum_{(k,p,q) \in T_\omega (a,b) }\,    \varepsilon(\omega, k)   \,\varepsilon_p(a)  \,\varepsilon_q(b) \langle a_pb_q \rangle. \end{equation}

   \begin{theor}\label{1aee++++}  The skew-symmetric bracket $[-, -]_X$ in~$M$ defined by
$$
[x,y]_X  =   
 [x,y]_{X,\omega}  -  [y,x]_{X,\omega}  $$
 for all $x,y \in M$ and a gate orientation~$\omega$ of~$X$  does not depend on~$\omega$.
 \end{theor}

We  prove this theorem in Section~\ref{p45} using the content of Section~\ref{p44}.
 We call   the bracket $[-,-]_X$  the \emph{homotopy intersection  bracket of~$X$}. Both   brackets $[-,-]_{X,\omega}$ and $[-,-]_{X}$    generalize Goldman's  bracket (\cite{Go1}, \cite{Go2}): 
 the value of   $[-,-]_{X,\omega}$  (respectively,  $[-,-]_{X}$)  on any pair of  free homotopy  classes  of loops in $\Sigma'\subset X$    is equal to  their  Goldman's    bracket    (respectively,  twice this bracket).    The free homotopy  classes  of loops lying in $Y\subset X$  annihilate  both $[-,-]_{X,\omega}$ and $[-,-]_{X}$.

\subsection{The pairing $\mu_k$}\label{p44} For each $k\in \pi_0$, we define a bilinear form $\mu_k:M\times M\to M$ as follows. It suffices to define $\mu_k(x,y)$ for all $x,y \in  \mathcal L$. To this end,     pick  generic  loops $a,b$ representing $x,y$, and set
$$\mu_k (x,y)=\sum_{p\in a \cap \alpha_k ,q \in b \cap \alpha_k}\,    \varepsilon_p(a)  \,\varepsilon_q(b) \langle a_pb_q \rangle \in M. $$

\begin{lemma}\label{1aeee+cvb+}   The vector $\mu_k (x,y)  \in M$ does not depend on the choice of generic  loops $a,b$ representing $x,y$.
\end{lemma}

\begin{proof} We  need only  to prove that $\mu_k (x,y) $ is preserved under the  loop moves on $a,b$.
For the moves $L_0 - L_3$, this is clear from the   definitions. A move $L_4$ on~$a$ creates two additional   points $p', p'' \in a\cap  \alpha_k$ for some~$k$  such that $\varepsilon_{p'}(a) =
-\varepsilon_{p''}(a) $. Then the expressions $\varepsilon_{p'}(a)  \,\varepsilon_q(b) \langle a_{p'}b_q \rangle$ and $\varepsilon_{p''}(a)  \,\varepsilon_q(b) \langle a_{p''}b_q \rangle$ cancel each other for all $q \in b \cap \alpha_k$. So, this move preserves  $\mu_k (x,y)$.  The move $L_5$ on~$a$ replaces two points $p_1, p_2 \in a \cap \alpha_k$ by two points $p'_1, p'_2$ such that $\varepsilon_{p'_i}(a)=\varepsilon_{p_i}(a)$ and $\langle a_{p'_i} b_q \rangle =\langle a_{p_i}b_q \rangle $ for $i=1,2$ and any $q \in b \cap \alpha_k$. 
So, this move preserves  $\mu_k (x,y)$. Similar computations show  that   the moves $L_4, L_5$ on~$b$ preserve  $\mu_k (x,y)  $. 
\end{proof}

Lemma \ref{1aeee+cvb+} shows   that $\mu_k (x,y)$ depends only on $x,y$. The  identity $\langle a_pb_q \rangle=\langle b_q a_p \rangle$ implies that  $\mu_k(x,y)=\mu_k(y,x)$ for all $x,y \in M$.

\subsection{Proof of Theorem~\ref{1aee++++}}\label{p45} It suffices to prove that  \begin{equation}\label{cvbbbnhoribulle} [x,y]_{X, {l}\omega}  -  [y,x]_{X,{l}\omega}  =[x,y]_{X,\omega}  -  [y,x]_{X,\omega} \end{equation}
for all $x,y\in \mathcal L$ and ${l}\in \pi_0$.     Pick an $\omega$-admissible  pair of loops $a,b$ representing respectively $x,y$. We  compute $ [x,y]_{X, \omega}$ from  \eqref{deff}   and compute $ [x,y]_{X, {l}\omega} $ applying~\eqref{chorieee++} (with $\omega$ replaced by ${l}\omega$) to the  pair $a,b$.   The same argument as in the proof of Theorem~\ref{1aee} shows that 
\begin{equation}\label{tyr} [x,y]_{X, {l}\omega} -  [x,y]_{X, \omega}    =
  \sum_{p\in a \cap \alpha_{{l}} ,q\in b \cap \alpha_{{l}}}\, \varepsilon({l}\omega, {l})\,  \varepsilon_p(a) \,  \varepsilon_q(b) \langle a_pb_q \rangle \end{equation}
 $$= 
 - \varepsilon(\omega, {l}) \,  \mu_{{l}} (x,y).$$
Since the form $ \mu_{{l}}$ is symmetric,   the expression $[x,y]_{X, {l}\omega}  -   [x,y]_{X,\omega}$
is symmetric in $x, y$.  This  implies   \eqref{cvbbbnhoribulle}  and completes the proof  of the theorem.

\subsection{Computation of $[-,-]_X$}\label{The edbdbdforccm bullet}   To compute $[x,y]_X$ for  $x,y\in \mathcal L$ we will use a method parallel to the one   in Section~\ref{The edbdbdform bullet}. Pick a generic pair of loops $a,b$ in~$X$ representing $x,y$. Let   $\omega_0$ be the  orientation of the gates induced by  the orientation of $\Sigma'\subset \Sigma$.  Lemma~\ref{1aeee++}  implies that 
  $$
[x,y]_X = [x, y]_{X, \omega_0}    - [y,x]_{X,\omega_0}     $$
$$ = 2 \sum_{r\in a\cap b} \varepsilon_r  (a,b) \langle a_r b_r \rangle    +\sum_{(k,p,q) \in T(a,b) }\,   \delta  (p,q)    \,\varepsilon_p(a)  \,\varepsilon_q(b) \langle a_pb_q \rangle $$
where $ \delta  (p,q)=\pm 1$ is defined     in Section~\ref{The edbdbdform bullet}.     If $a\cap b =\emptyset$, then
\begin{equation}\label{chorvvieee++veryryrsss}
[x,y]_X =  \sum_{(k,p,q) \in T(a,b) }\,   \delta  (p,q)    \,\varepsilon_p(a)  \,\varepsilon_q(b)  \langle a_pb_q \rangle  . \end{equation}

  \subsection{Remarks}\label{The dbdbdform bullet}   
 
 1. Applying \eqref{tyr} consecutively  to all $l\in \pi_0$ and using~\eqref{chorisdfbulle}, we obtain that for any $x,y\in M$ and any gate orientation~$\omega$ of~$X$,  
$$
[x,y]_{X, \omega} +  [y,x]_{X, \omega} =\sum_{l\in \pi_0}  \varepsilon(\omega, l) \,  \mu_{l}(x,y).  $$
Since $[x,y]_X=[x , y]_{X, \omega} -[y,x ]_{X, \omega} $, we deduce that 
$$
2 [x , y]_{X, \omega} = [x,y]_X +\sum_{l\in \pi_0}  \varepsilon(\omega, l) \,  \mu_{l}(x,y)  .$$
As a consequence,  the   form $\bullet_{X,\omega}$ in $H_1(X)$ may be computed from    $\bullet_{X}$  via 
$$2\,x \bullet_{X, \omega} y    =x  \bullet_{X} y +\sum_{l\in \pi_0}  \varepsilon(\omega, l) \,  v_l(x) \, v_l(y)$$
for all $x,y \in H_1(X)$. 

2. Generally speaking, the brackets $[-,-]_{X, \omega}$ and $[-,-]_X$ do not satisfy the Jacobi identity. Their Jacobiators can be computed in terms of operations associated with the gates of~$X$. Similar results hold for the self-intersection cobrackets defined in the next section;  this will be discussed  in more detail elsewhere.

   \section{The cobrackets}
   
   We define self-intersection cobrackets for loops in the quasi-surface~$X$. 
    
     \subsection{The cobracket $\nu_{X,\omega}$}\label{cobob} 
 For a loop~$a$ in~$X$, we set $\langle a \rangle_0=\langle a \rangle \in   M=M(X)$  if~$a$ is non-contractible and $\langle a \rangle_0=0\in M$ if~$a$ is contractible.   
A generic loop~$a$ crosses each point  $r\in \# a$ twice; we let $v^1_r, v^2_r$ be the tangent vectors of~$a$ at~$r$ numerated so    that the pair $(v^1_r, v^2_r)$ is positively oriented. For $i=1,2$,   let $a^i_r$ be the loop   starting in~$r$ and going  along~$a$   in the direction of the vector $v^i_r$    until the first return to~$r$. Up to parametrization, $a=a^1_ra^2_r$ is the product of the  loops $a^1_r, a^2_r$ based at~$r$.  
Also, for any $k\in \pi_0$ and  any  distinct points
$p_1, p_2 \in a \cap \alpha_k$ we define a loop $a_{p_1, p_2}$ in~$X$ which goes from $p_1$  to $p_2$ along~$a$ and   then goes back  to~$p_1 $  along the gate~$\alpha_k$. 
  Consider the set of ordered triples $$T(a)= \{ (k\in \pi_0,p_1 \in a \cap \alpha_k , p_2 \in a \cap \alpha_k) \,  \vert \,
p_1 \neq p_2\}.$$    We  call  a   triple   $ (k,p_1, p_2) \in  T(a) $ a  \emph{chord of~$a$ with endpoints} $p_1, p_2$.

For a gate orientation~$\omega$ of~$X$, we  let $T_\omega(a)  $ be the set of   chords $(k,p_1, p_2)\in  T(a)   $  such that $p_1<_\omega p_2$. Set $$\nu_{X,\omega}(a)=\sum_{r\in \# a} (\langle a^1_r \rangle_0 \otimes \langle a^2_r \rangle_0  -\langle a^2_r \rangle_0  \otimes \langle a^1_r \rangle_0)$$
$$
 +\sum_{(k,p_1, p_2)\in T_\omega(a)} \, \varepsilon(\omega, k) \,  \varepsilon_{p_1}(a) \, \varepsilon_{p_2}(a) \,    \langle  a_{{p_2},{p_1}}\rangle_0 \otimes \langle  a_{{p_1},{p_2}} \rangle_0 \in  M  \otimes M.$$

\begin{lemma}\label{svvvssvtrbucte}  The cobracket  $\nu_{X,\omega}(a)$   is preserved under all loop moves on~$a$.
\end{lemma}  
  
\begin{proof} The moves $L_0-L_3$ proceed in $\Sigma'$ and are treated as in \cite{Tu1}. The move $L_4$ pushes  a  branch of~$a$    across a gate $\alpha_{{l}}$ for some ${l} \in \pi_0$ creating a loop $a'$ which has  two additional  crossings ${q_1},{q_2} \in a' \cap \alpha_{{l}}$ such that $q_1<_\omega q_2$ and $\varepsilon_{q_1} (a)=-\varepsilon_{q_2}(a)$.  The   contributions to the cobracket  of the self-intersection points   and  of   the  chords  containing neither~$q_1$ nor~$q_2$  are the same before and after the deformation. The   contributions to $\nu_{X,\omega}(a')$ of the chords  containing exactly one of the points ${q_1},{q_2}$ cancel each other.   The   chord      $({l},{q_1},{q_2})$ of~$a'$   contributes zero to  $\nu_{X,\omega}(a')$  because at least  one of the    loops $a_{{q_1},{q_2}}$ and $ a_{q_2, q_1}$ is contractible.   Therefore
 $\nu_{X,\omega}(a)=\nu_{X,\omega}(a')$. We now prove the invariance   of $\nu_{X,\omega}(a)$ under the move $L_5$ which  pushes a self-crossing of~$a$ in $\Sigma'$ across a gate, say $\alpha_{{l}}$, into $X\setminus \Sigma'$.  Note   that under  the inversion of the orientation of~$\Sigma$,  the signs $\varepsilon( \omega, k) $ and the expression  $\nu_{X,\omega}(a)$ are  multiplied by~$-1$. Therefore, inverting if necessary  the given  orientation of~$\Sigma$, we can reduce the proof of the invariance of $\nu_{X,\omega}(a)$ under our move to the case where $\varepsilon( \omega, {l})=+1$.  Assume that the  move changes~$a$ in a small disk~$D$ by  pushing  a self-crossing  $r_0 \in \#a$   from $  D \cap \Sigma'$   to $D \setminus \Sigma'$. For $i=1,2$ set $\gamma_i=\langle a^i_{r_0}\rangle_0 \in M$. The branches of~$a$ meeting at $r_0$ intersect the gate $\alpha_{{l}}$ in two points lying in~$D$. We label these points $q_1,q_2$ so that $q_1<_\omega q_2$. 
 By definition,   $r_0$ contributes $\gamma_1\otimes \gamma_2  - \gamma_2\otimes \gamma_1 $  to $\nu_{X,\omega}(a)$.  The contribution of the chord $({l}, q_1,q_2)$  to $\nu_{X,\omega}(a)$ also  can be computed   from the definitions: it is equal to  $\gamma_2\otimes \gamma_1 $ if $\varepsilon_{q_1}(a)= \varepsilon_{q_2}(a)$ and  to  $- \gamma_1\otimes \gamma_2 $ otherwise. Thus the joint contribution of $r_0$ and  $({l}, q_1,q_2)$  to $\nu_{X,\omega}(a)$ is equal to 
  $\gamma_1\otimes \gamma_2 $ if $\varepsilon_{q_1}(a)= \varepsilon_{q_2}(a)$ and   to  $- \gamma_2\otimes \gamma_1 $  otherwise. The loop, $a'$, produced by the move meets $D\cap \alpha_{{l}}$ in two points forming a chord of~$a'$.
A similar  computation shows that the contribution of this chord    to $\nu_{X,\omega}(a')$ also  is  
  $\gamma_1\otimes \gamma_2 $ if $\varepsilon_{q_1}(a)= \varepsilon_{q_2}(a)$ and     $- \gamma_2\otimes \gamma_1 $  otherwise. All the other self-crossings and chords contribute   the same expressions to  $\nu_{X,\omega}(a)$  and $\nu_{X,\omega}(a')$. Therefore $\nu_{X,\omega}(a)=\nu_{X,\omega}(a')$.  \end{proof}
 
Lemma \ref{svvvssvtrbucte}   implies that  $\nu_{X,\omega}(a) \in M\otimes M$ depends only on the free homotopy  class $\langle a \rangle$  of~$a$. 
 The    map $\langle a \rangle \mapsto \nu_{X,\omega}(a) : \mathcal L \to M^{\otimes 2}$     extends  uniquely  to a   linear map $ 
M\to M^{\otimes 2}$   denoted~$\nu_{X,\omega}$.  


\subsection{The cobracket~$\nu_X$}\label{kocncncnkoeeata}    
  We   define a  cobracket  $\nu_X$  independent of~$\omega$.

   \begin{theor}\label{cobr1aee++++}  Let~$P$ be
 the linear  automorphism  of  $M \otimes M $ carrying $x\otimes y$ to $y\otimes   x$   for all $x,y \in M$.   The skew-symmetric cobracket 
$$
\nu_X  =   
 \nu_{X,\omega} - P \nu_{X,\omega}: M \to M \otimes M $$ 
does not depend on the choice of~$\omega$.
 \end{theor}
 
 \begin{proof} It suffices to prove that $ \nu_{X,\omega} - P \nu_{X,\omega}$  is preserved  when~$\omega$ is replaced with ${l}\omega$  for  ${l}\in \pi_0$. For any generic loop~$a$ in~$X$,  we have 
 $$\nu_{X, \omega}(a) -\nu_{X, {l} \omega}(a)=\varepsilon(\omega, {l}) \, \sum_{({l}, p_1, p_2)\in T_\omega(a)} \,   \varepsilon_{p_1}(a) \, \varepsilon_{p_2}(a) \,    \langle  a_{{p_2},{p_1}}\rangle_0 \otimes \langle  a_{{p_1},{p_2}} \rangle_0 $$
 $$-\, 
 \varepsilon({l} \omega, {l}) \, \sum_{({l}, p_1, p_2)\in T_{{l}\omega}(a)} \,     \varepsilon_{p_1}(a) \, \varepsilon_{p_2}(a) \,   \langle  a_{{p_2},{p_1}}\rangle_0 \otimes \langle  a_{{p_1},{p_2}} \rangle_0.$$
 Clearly,  $  \varepsilon({l} \omega, {l}) =- \varepsilon(\omega, {l})$. Also,    the inclusion  $({l}, p_1, p_2)\in  T_{{l}\omega} (a)   $ holds  if and only if $({l},p_2,p_1)\in  T_{\omega} (a)   $. Therefore
  $$\nu_{X, \omega}(a) -\nu_{X, {l} \omega}(a)=$$
  $$=\varepsilon(\omega, {l}) \, \sum_{({l}, p_1, p_2)\in T_\omega(a)} \,  \varepsilon_{p_1}(a) \, \varepsilon_{p_2}(a) \,  \big ( \langle  a_{{p_2},{p_1}}\rangle_0 \otimes \langle  a_{{p_1},{p_2}} \rangle_0 +\langle  a_{{p_1},{p_2}}\rangle_0 \otimes \langle  a_{{p_2},{p_1}} \rangle_0  \big ). $$
  The latter expression  is, obviously, invariant under the  transposition~$P$. So, 
  $$\nu_{X, \omega}(a) -\nu_{X, {l} \omega}(a)=P\nu_{X, \omega}(a) - P \nu_{X, {l} \omega}(a)$$
or, equivalently,  $$\nu_{X, \omega}(a) -P\nu_{X, \omega}(a) = \nu_{X, {l} \omega}(a)- P \nu_{X, {l} \omega}(a).$$
 \end{proof}
 
We call $\nu_X$   the \emph{self-intersection  cobracket}  of~$X$. Both  cobrackets  $\nu_{X,\omega}$ and $\nu_{X}$    generalize the
 cobracket $\nu$  defined for loops in surfaces in \cite{Tu1}: 
 the value of   $\nu_{X,\omega}$  (respectively,  $\nu_{X}$)  on any   free homotopy  class  of  loops in $\Sigma'\subset X$    is equal to  the value of~$\nu$ on this class  (respectively,  twice that value).    The   free homotopy  classes  of loops lying in $Y\subset X$  are annihilated  by both $\nu_{X,\omega}$ and $\nu_{X}$.

\subsection{Computation of $\nu_X$}\label{The edbdxxxbdforccm bullet}   To compute $\nu_X $ we  use a method parallel to the one  used  in Sections~\ref{The edbdbdform bullet} and~\ref{The edbdbdforccm bullet}. Namely, for any generic  loop~$a $ in~$X$, we have 
  $$
\nu_X(  \langle a  \rangle )   = 2 \sum_{r\in \# a} (\langle a^1_r \rangle_0 \otimes \langle a^2_r \rangle_0  -\langle a^2_r \rangle_0  \otimes \langle a^1_r \rangle_0)$$
$$
+ \sum_{(k,p_1,p_2)\in T (a)} \, \delta(p_1,p_2) \,    \varepsilon_{p_1}(a) \, \varepsilon_{p_2}(a) \,      \langle  a_{{p_1},{p_2}}\rangle_0 \otimes \langle  a_{{p_2},{p_1}} \rangle_0 . $$
      If $\# a =\emptyset$, then
\begin{equation}\label{choricvxvxeee++veryryrsss}
\nu_X(  \langle a  \rangle  )   =  
 \sum_{(k,p_1,p_2)\in T (a)}  \, \delta(p_1,p_2) \,     \varepsilon_{p_1}(a) \, \varepsilon_{p_2}(a) \,    \langle  a_{{p_1},{p_2}}\rangle_0 \otimes \langle  a_{{p_2},{p_1}} \rangle_0  . \end{equation}

 \subsection{Examples}\label{A vanishing case} We give  two examples where the  operations 
 $\bullet_X, [-,-]_X$, and $ \nu_X $  vanish. Set $I=[0,1]$ and  $\Sigma =I^2$ with an arbitrary orientation.
 
 1.   Let $\alpha=I \times \{0\} \subset \partial \Sigma$ and  $\Sigma'= I \times  [1/3, 1] \subset \Sigma$. Then~$X$ has only one gate $ I \times \{1/3\}$.  It is clear  that any loop in~$X$ can be deformed away from~$\Sigma'$. Therefore $\bullet_{X, \omega}=0, [-,-]_{X, \omega}=0$, and $ \nu_{X, \omega} =0$ for both gate orientations~$\omega$ of~$X$. 
 Consequently,    $\bullet_X=0$, $[-,-]_X=0$, and $\nu_X=0$. 
 
 2. Let  $\alpha=I \times \{0,1\} \subset \partial \Sigma$ and  $\Sigma'= I \times  [1/3, 2/3] \subset \Sigma$. Then~$X$ has two gates   $\alpha_1=I \times \{1/3\}$ and $\alpha_2=I \times \{2/3\}$. Let~$\omega$ be  the   orientation of $\alpha_1, \alpha_2$ induced by the orientation of~$I$   from~$0$ to~$1$.  Any pair of  free homotopy classes of loops  in~$X$ can be represented by   loops $a,b$ such that~$a$ meets~$  \Sigma'$ at several   segments $\{s\} \times[1/3, 2/3]$ with  $s\in [0,1/3]$ and~$b$ meets~$  \Sigma'$ at several   segments  $\{t\} \times[1/3, 2/3]$ with $t\in [2/3, 1]$. Then the pair $a,b$ is $\omega$-admissible and $a\cap b=\emptyset$. Hence,  $a
\bullet_{X, \omega} b =0$. Consequently, $\bullet_{X, \omega}=0$ and   $\bullet_X=0$. 
Similar  arguments  show that  $[-,-]_X=0$. 
Next,  any free homotopy class  of  loops  in~$X$ can be represented by a  generic loop  $a$ which meets~$  \Sigma'$ at  the 
   segments $\{s\} \times[1/3, 2/3]$ where~$s$ runs over a finite set   $S \subset  (0,1)$.
Clearly,  $\#a= \emptyset$. The set $T_\omega(a)$ consists of the  triples $(k\in \{1,2\},s_1,s_2)$   with  $s_1,s_2\in S$ and $s_1<s_2$. For any such $s_1,s_2$, the triples $(1,s_1,s_2)$ and $(2, s_1,s_2)$ contribute opposite values to $\nu_{X, \omega} (a)$ because  $\varepsilon(\omega, 1)= - \varepsilon(\omega, 2)$ while all  other terms of these  contributions are   the same. Thus,  $\nu_{X, \omega}(a)=0$. Consequently, $\nu_{X, \omega}=0$ and $\nu_X=0$. Note that for the  gate orientation~$\omega$ of~$X$ which directs  one    gate from~$0$ to~$1$ and   the other gate  from~$1$ to~$0$, the operations  $\bullet_{X, \omega}, [-,-]_{X, \omega}$, and $ \nu_{X, \omega} $ may be non-zero.

  \section{Transformations of quasi-surfaces}\label{Operations on quassi-surfaces}
       
We study    two transformations of the quasi-surface~$X$:  the  transformation~D (for disjoint unions) and the  transformation~C (for cuttings). Both~D and~C  preserve the underlying topological space of~$X$ but change the  structure of a quasi-surface.

    \subsection{The  transformation D}\label{modiffD}   The  transformation D  applies when $\Sigma=\sqcup_{j=1}^N \Sigma_j$  is a disjoint union of $N\geq 2$ oriented surfaces.  For each $j=1,..., N$, we let $Y_j$ be the topological space obtained by  glueing $N-1$  surfaces $\{\Sigma_i\}_{i\neq j}$ to~$Y$ along the   maps $\alpha\cap \partial \Sigma_i \to Y$ used in the definition of~$X$.  The  underlying  topological space of~$X$ is obtained by    glueing  $\Sigma_j$ to~$Y_j$ along the  map  $\alpha\cap \partial \Sigma_j \to Y \subset Y_j $ used in the definition of~$X$. This turns the   space in question   into   a quasi-surface,  $X_j$,   with surface core~$ \Sigma_j$ and singular core~$Y_j$.   For the reduced surface core   and the gates  of~$X_j$ we take  $\Sigma'\cap \Sigma_j$ and the gates of~$X$ lying  in $\Sigma_j$.

\begin{lemma}\label{ri1a}     We have 
\begin{equation}\label{choribnbn++}
\bullet_{X}   =  \sum_{j=1}^N   \bullet_{X_j}  : H_1(X) \times H_1(X) \to \ZZ,   \end{equation}
 \begin{equation}\label{choribnbn++br}
 [-,-]_{X}   =  \sum_{j=1}^N     [-,-]_{X_j}  : M\times M \to M,  \end{equation}
 \begin{equation}\label{choribnbn++nu}
\nu_{X}   =  \sum_{j=1}^N    \nu_{X_j}  : M \to M \otimes M.  \end{equation}
 \end{lemma}

\begin{proof} Note that each gate orientation~$\omega$ of~$X$ restricts to a gate orientation $\omega_j$ of $X_j$. Formulas \eqref{choribnbn++}--\eqref{choribnbn++nu}  are direct consequences of  the following stronger claim: for any~$\omega$,  we have
\begin{equation}\label{choribnbnomega}
 \bullet_{X, \omega}    = \sum_{j=1}^N   \bullet_{X_j, \omega_j}  : H_1(X) \times H_1(X) \to \ZZ. \end{equation}
   \begin{equation}\label{choribnbn++bromega}
 [-,-]_{X, \omega}   =  \sum_{j=1}^N     [-,-]_{X_j, \omega_j}  : M\times M \to M,  \end{equation}
 \begin{equation}\label{choribnbn++omega-}
\nu_{X, \omega}   =  \sum_{j=1}^N    \nu_{X_j, \omega_j}  : M \to M \otimes M.  \end{equation}
To prove \eqref{choribnbnomega}, we pick an $\omega$-admissible pair $a,b$ of loops in~$X$ representing $x,y\in H_1(X)$.  Then $ x \bullet_{X, \omega} y $ is the algebraic number of intersections of $a,b$ in $\Sigma'= \sqcup_j (\Sigma' \cap \Sigma_j)$. Also,  for all~$j$,  the  pair of loops  $a,b$ is    $\omega_j$-admissible in $X_j$  and  $x \bullet_{X_j, \omega_j} y  $
 is the algebraic number of intersections of  $a,b$    in  $\Sigma' \cap \Sigma_j$. Hence, 
 $$x \bullet_{X, \omega} y  = \sum_{j=1}^N   x \bullet_{X_j, \omega_j} y .$$
 The proofs of   \eqref{choribnbn++bromega},~\eqref{choribnbn++omega-} are similar.
\end{proof}

 \subsection{The  transformation C}\label{modiffC}  A  submanifold~$\beta$ of  a manifold $N$ is  said to be  \emph{proper} if $\beta \cap \partial N=\partial \beta$.  The  transformation
 C
  applies when   we are given  a  proper compact  1-dimensional submanifold~$\beta$ of~$ \Sigma'  $ whose     components       are   segments   disjoint from  the gates of~$X$ (which, recall, all lie in $\partial \Sigma'$). 
   Cutting $\Sigma \supset \Sigma'  \supset \beta$ along~$\beta $, we obtain an   oriented surface~$\Sigma^\beta$. A copy of the 1-manifold $\alpha \subset \partial \Sigma \setminus \Sigma'$   lies in  $\partial \Sigma^\beta$ and is denoted~$\alpha^\beta$. The 1-manifold~$\beta$ gives rise to two copies of itself  in $ \partial \Sigma^\beta \setminus \alpha^\beta$.  The  underlying topological space of~$X$  can be obtained by glueing  the surface   $\Sigma^\beta$ to   the disjoint union $Y^\beta=Y \sqcup   \beta$ along the map $\alpha^\beta=\alpha \to Y \subset Y^\beta$ used in the definition of~$X$ and along  the tautological identity maps of the copies of~$\beta $ in $\partial \Sigma^\beta$ to   
  $ \beta  \subset   Y^\beta $.  
  This turns the underlying topological space of~$X$  into a quasi-surface $X^\beta$ with  surface core  $\Sigma^\beta$ and singular core $Y^\beta$.

   \begin{lemma}\label{ri1a}    We have
   \begin{equation}\label{choribnbn++C}
\bullet_{X}   = \bullet_{X^\beta}  : H_1(X) \times H_1(X) \to \ZZ,   \end{equation}
 \begin{equation}\label{choribnbnr++C}
 [-,-]_{X}   =      [-,-]_{X^\beta}  : M\times M \to M,  \end{equation}
 \begin{equation}\label{choribnbnm++C}
\nu_{X}   =      \nu_{X^\beta}  : M \to M \otimes M.  
  \end{equation} \end{lemma}

\begin{proof} The   gates  of $X^\beta$ are the gates of~$X$ and   additional  gates associated with the components  $\{\beta_l\}_l$ of~$\beta$. Namely, each $\beta_l$ gives rise to two  gates of $X^\beta$ which are proper segments in~$\Sigma'$   running   \lq\lq parallel''  to $\beta_l$  on different sides of $\beta_l$ in~$\Sigma'$.  For each~$l$, fix  an  orientation of~$\beta_l$ and orient the  associated   gates so that they look in the same direction as $\beta_l$. Then every gate orientation~$\omega$ of~$X$ determines  a gate orientation $\omega^\beta$ of~$X^\beta$.
Formulas \eqref{choribnbn++C}--\eqref{choribnbnm++C}  are   consequences of  the following stronger claim: for any~$\omega$,  we have
\begin{equation}\label{choribnbn}
 \bullet_{X, \omega}  =    \bullet_{X^\beta, \omega^\beta}   : H_1(X) \times H_1(X) \to \ZZ. \end{equation}
  \begin{equation}\label{choribnbnr++Cs}
 [-,-]_{X, \omega}   =      [-,-]_{X^\beta, \omega^\beta}  : M\times M \to M,  \end{equation}
 \begin{equation}\label{choribnbnm++Cs}
\nu_{X, \omega}   =      \nu_{X^\beta, \omega^\beta}  : M \to M \otimes M.  
  \end{equation}
  To prove \eqref{choribnbn}, pick an $\omega$-admissible pair $a,b$ of loops in~$X$ representing   $x,y\in H_1(X)$. Deforming if necessary $a,b$ near~$\beta$ we can assume that $a, b$ are transversal to~$\beta$,   all crossings  of~$a$ with~$\beta$ lie near the  tails of the components, and all crossings  of~$b$ with~$\beta$ lie near the  heads of the components.  Then the  pair $a,b$ is    $\omega^\beta$-admissible. The integer $x \bullet_{X, \omega} y $ is the algebraic number of intersections of $a,b$ in $\Sigma'$.  Since all these intersections lie away from $\beta$ and from the gates of~$X$, they bijectively correspond to   the intersections of $a,b$    in the reduced surface core   of $X^\beta$ (and have the same signs). Therefore  $x \bullet_{X, \omega} y  = x \bullet_{X^\beta, \omega^\beta} y $.
  The proofs of the equalities \eqref{choribnbnr++Cs} and~\eqref{choribnbnm++Cs} are similar. 
\end{proof}

\section{Stars in quasi-surfaces}\label{Stars}

We state and prove analogues of Theorems \ref{s-inter}--\ref{s-inter++cobr} for quasi-surfaces.

\subsection{Stars in~$X$}\label{Stars++} By a \emph{star} in the quasi-surface~$X$   we mean a  star~$s$ embedded in the surface core~$\Sigma$ of~$X$ so that $\partial s=s \cap \partial \Sigma  \subset \partial  \Sigma \setminus \alpha$.  We derive from~$s$    a new  structure of a quasi-surface in the underlying topological space of~$X$. Denote the number of leaves of~$s$   by $\vert s \vert$.  Pick a closed regular neighborhood~$V$ of~$s$ in $\Sigma \setminus \alpha \subset X$ and provide~$V$ with orientation induced by that of~$\Sigma$. 
It is clear that~$V$ is a 2-disk whose boundary  is formed by $\vert s \vert$ disjoint segments    in $\partial \Sigma \setminus \alpha$ and $\vert s \vert$ disjoint proper segments $\beta^s_1,..., \beta^s_{\vert s \vert}$ in~$\Sigma$. Set $Y_\star=\overline {X\setminus V} \subset X$ where the overline stands for the closure in the  underlying topological space of~$X$.     Taking~$V$ as the  surface core,   $Y_\star$ as the   singular core and  glueing~$V$ to $Y_\star$ along the   inclusions 
  $\{\beta_i^s  \hookrightarrow  Y_\star\}_{i=1}^{\vert s \vert  }$ we obtain a quasi-surface, ${X_\star}={X_\star}(s)$, with  the same underlying topological space as~$X$. This allows us to consider the bilinear maps
  $$\bullet_s= \bullet_{{X_\star}}: H_1(X) \times H_1(X) \to \ZZ, \quad  [-,-]_{s}=[-,-]_{{X_\star}}: M \times M\to M
  $$ and  the linear map $\nu_s=\nu_{{X_\star}}: M \to M \otimes M$.

  We next  compute the form $\bullet_s$  via   intersections of loops  with~$s$. As in Section~\ref{Introduction}, the orientation  of~$\Sigma$  at the  center  of~$s$ determines a cyclic order  in the set ${\text {Edg}}(s)$ of   edges of~$s$. For   $e \in {\text {Edg}}(s)$   we let $e^+ \in {\text {Edg}}(s)$ be  the next   edge   with respect to this order.  We say that a loop   in~$X$ is \emph{$s$-generic} if it     misses the vertices of~$s$,   meets all  edges of~$s$  transversely, and never  traverses a point of~$s$ more than once. A family of loops in~$X$  is \emph{$s$-generic} if   these loops   are $s$-generic and do not   meet at  points of~$s$.  
      Given an $s$-generic loop~$a$ in~$X$, we    let $a\cap s$ be the set of points of~$s$ traversed by~$a$. For an edge   $e \in {\text {Edg}}(s)$,  set $a\cap e =(a\cap s) \cap e$.  For $p\in a\cap e$,   the intersection sign of~$a$ and~$e$ at~$p$ is denoted $\mu_p (a) $ (recall that the  edges of~$s$ are directed from the center of~$s$ to the leaves).
  The integer $a\cdot e =\sum_{p\in a \cap e}\mu_p(a)$ is the algebraic number of   intersections of~$a$ and~$e$.

   \begin{lemma}\label{rigovbvbq1a}    For any   $s$-generic  pair of loops  $a,b$ in~$X$ representing $x,y \in   H_1(X) $,  
  \begin{equation}\label{homovbvbl}  x \bullet_s y = \sum_{e \in {\text {Edg}}(s)} \big (( a\cdot e) (b\cdot e^+) - (b\cdot e) ( a\cdot e^+) \big ). 
   \end{equation} 
 \end{lemma}

\begin{proof}    Consider the quasi-surface $X_\star=X_\star(s) $  derived as   above from a closed regular neighborhood~$V$ of~$s$ in $\Sigma \setminus \alpha$.   For  the reduced surface core of $X_\star$ we take a smaller closed regular neighborhood  $V'\subset V$ of~$s$.  The  gates of ${X_\star}$ are the segments in $\partial V'$ separating $V'$ from the rest of~$V$. Moving along the circle $\partial V'$ in the direction determined  by the orientation of~$V'$ induced by that of~$\Sigma$, we meet all gates in a certain cyclic order. For a gate $K$ of  ${X_\star}$, denote the next gate with respect to this cyclic order  by $K^+$. Note   that   there is a unique edge $e=e_K$ of~$s$ such that~$K$ is obtained by pushing the segment $e\cup e^+$ into $\Sigma \setminus s$.
Clearly, $e_{K^+}=(e_K)^+$.

To proceed, we  select the   closed regular neighborhood $V'\subset V$  taking into account the loops $a,b$.
We say that~$V'$   is \emph{$a$-adapted} if the set  $a(S^1)\cap V'$  is a disjoint union  of      proper segments $\{f_p\}_{p\in a\cap s}$  in~$V'$ such that $p\in f_p$  for all $p\in a\cap s$.
We  denote   the endpoints of the segment ~$f_p$  by $p', p''$  so that  the pair  (the orientation of $f_p$  from $p'$ to $p''$,  the orientation of the edge of~$s$ containing~$p$) determines the given orientation of~$\Sigma$ at~$p$. 
If  $V'$ is  $a$-adapted, then~$a$ has no self-intersections in~$V'$ and   crosses the gates precisely  at the points $\{p',p''\}_{p\in a\cap  s}$. The  crossing signs of~$a$ with the gates (see Section~\ref{fofofor}) are computed by  \begin{equation}\label{KK0} \varepsilon_{p'}(a)= \mu_p(a)  \quad {\text{and}}   \quad \varepsilon_{p''}(a)=  - \mu_p(a)   \end{equation}   for all $p\in a\cap s$.  
We select   the neighborhood  $V' \subset V$  of~$s$ so small (= narrow),   that it  is   $a$-adapted, $b$-adapted, and the loops $a,b$ do not meet in~$V'$.
Then for any gate~$K$ of $X_\star$, we have 
\begin{equation}\label{KK1}
a\cap K=\{p'\, \vert \, p \in a \cap e_K\} \sqcup \{p'' \, \vert \, p \in a\cap (e_K)^+\}  
  \end{equation}
  and, similarly, 
\begin{equation}\label{KK2} b\cap K=\{q'\, \vert \, q \in b \cap e_K\} \sqcup \{q'' \, \vert \, q \in b\cap (e_K)^+\}. \end{equation} 
We compute  $x\bullet_{{X_\star}} y$ via~\eqref{chorieee+bbb+veryryrsss}. The sum on the right hand-side of~\eqref{chorieee+bbb+veryryrsss} runs over all triples (a gate~$K$ of $X_\star$, a point of $a\cap K$, a point of $b\cap K$). By~\eqref{KK1} and~\eqref{KK2}, the contribution of  such triples with fixed~$K$ is the sum of the following four  expressions: 
$$\sigma^1_K=\sum_{ p \in a \cap e_K ,q \in b \cap e_K }  \, \delta(p', q') \,\varepsilon_{p'}(a)  \,\varepsilon_{q'}(b) ,$$
$$\sigma^2_K= \sum_{  p \in a \cap e_K ,q \in b \cap (e_K)^+}    \, \delta(p', q'') \,\varepsilon_{p'}(a)  \,\varepsilon_{q''}(b),$$
$$\sigma^3_K= \sum_{  p \in a\cap (e_K)^+ ,q \in b \cap e_K}  \, \delta(p'', q')  \,\varepsilon_{p''}(a)  \,\varepsilon_{q'}(b),$$
$$\sigma^4_K= \sum_{  p \in a\cap (e_K)^+ ,q \in b \cap (e_K)^+}  \, \delta(p'', q'')   \,\varepsilon_{p''}(a)  \,\varepsilon_{q''}(b).$$
To simplify these expressions we use the following notation: for distinct points $p,q$ of an  edge of~$s$,   set $\mu(p,q)=1$ if~$p$ lies between  the center of~$s$ and~$q$ and   set  $\mu(p,q)=-1$ otherwise. For any  $e\in {\text {Edg}} (s)$ and any points $ p \in a \cap e ,q \in b \cap e $,  we have $\delta(p', q')=\mu(p,q)$ and $\delta(p'', q'')=-\mu(p,q)$. Using this   and~\eqref{KK0}, we  get 
$$\sigma^1_K=  \sum_{ p \in a \cap e_K ,q \in b \cap e_K }  \, \mu(p,q) \, \mu_{p}(a)  \,\mu_{q}(b) $$
and
$$\sigma^4_K=- \sum_{  p \in a\cap (e_K)^+ ,q \in b \cap (e_K)^+}  \,\mu(p,q)   \, \mu_{p}(a)  \,\mu_{q}(b).$$
When~$K$ runs over all gates, both $e_K$ and $(e_K)^+$ run over all edges of~$s$. Therefore $\sum_K  (\sigma^1_K + \sigma^4_K)=0$.
Observe next  that $\delta(p', q'')=-1$ for all $p \in a \cap e_K $ and $q \in b \cap (e_K)^+$. Hence, 
$$\sigma^2_K=   \sum_{  p \in a \cap e_K ,q \in b \cap (e_K)^+}    \,\mu_p(a) \,\mu_q(b)= (a\cdot e_K) (b\cdot (e_K)^+).$$
Similarly, $\delta(p'', q')=1$ for all $p \in a\cap (e_K)^+ ,q \in b \cap e_K$ and so 
$$\sigma^3_K =   - \sum_{  p \in a\cap (e_K)^+ ,q \in b \cap e_K}     \,\mu_p(a) \,\mu_q(b)=-  (a\cdot (e_K)^+) (b\cdot e_K).$$
Therefore  $$
x \bullet_s y = \sum_K  (\sigma^1_K+\sigma^2_K + \sigma^3_K+\sigma^4_K)= \sum_{e \in {\text {Edg}}(s)} \big (( a\cdot e) (b\cdot e^+) - (b\cdot e) ( a\cdot e^+) \big ) .$$ 
\end{proof}

 Given two $s$-generic loops $a,b$ in~$X$ and points $p\in a\cap s, q\in b\cap s$, we  pick a  path~$c$ from~$p$ to~$q$ in~$s$ and write $a_pb_q$ for the loop $a_p c b_qc^{-1}$. Clearly, the   free homotopy class of this loop  in~$X$  does not depend on the choice of~$c$.  
   
     \begin{lemma}\label{rigovbvbq1++++a}    For any  $s$-generic loops  $a,b$ in~$X$, we have 
$$ [\langle a\rangle, \langle b \rangle]_s$$
$$ = \sum_{e \in {\text {Edg}}(s)} \big ( \sum_{p\in a\cap e, q\in b\cap e^+} \mu_p(a) \, \mu_q(b) \, \langle  a_p b_q \rangle  -  \sum_{ p\in a\cap e^+, q\in b\cap e}  \mu_p(a) \, \mu_q(b) \,   \langle      a_p b_q \rangle\big )  . 
$$
 \end{lemma}

\begin{proof}     We use the same $V'$ as in the proof of Lemma~\ref{rigovbvbq1a}  and apply Formula~\eqref{chorvvieee++veryryrsss} to compute $[\langle a\rangle, \langle b \rangle]_s$. The  expressions $\sigma^1_K,..., \sigma^4_K$ are  replaced  with the sums 
$$ \sum_{ p \in a \cap e_K ,q \in b \cap e_K }  \, \delta(p', q') \,\varepsilon_{p'}(a)  \,\varepsilon_{q'}(b)  \langle  a_p b_q \rangle ,$$
$$  \sum_{  p \in a \cap e_K ,q \in b \cap (e_K)^+}    \, \delta(p', q'') \,\varepsilon_{p'}(a)  \,\varepsilon_{q''}(b)  \langle  a_p b_q \rangle ,$$
$$  \sum_{  p \in a\cap (e_K)^+ ,q \in b \cap e_K}  \, \delta(p'', q')  \,\varepsilon_{p''}(a)  \,\varepsilon_{q'}(b)  \langle  a_p b_q \rangle ,$$
$$  \sum_{  p \in a\cap (e_K)^+ ,q \in b \cap (e_K)^+}  \, \delta(p'', q'')   \,\varepsilon_{p''}(a)  \,\varepsilon_{q''}(b)  \langle  a_p b_q \rangle .$$
The rest of the argument goes along the same lines as  the proof of Lemma~\ref{rigovbvbq1a}.
  \end{proof}

 Given an $s$-generic loop~$a$ in~$X$  and points $p_1,p_2 \in a\cap s $, we   write $a_{p_1,p_2}$ for the loop going from~$p_1$ to~$p_2$  along~$a$ and then going back to~$p_1$
along a path in~$s$.   Clearly, the   free homotopy class of this loop   does not depend on the choice of the latter path.

     \begin{lemma}\label{rigovbvbq1++++abbbn}    For any  $s$-generic loop~$a$ in~$X$, we have 
$$ \nu_s(\langle a \rangle)   $$
$$=\sum_{e \in {\text {Edg}}(s)} \,  \sum_{p_1\in a\cap e, p_2\in a\cap e^+} \mu_{p_1}(a)  \,\mu_{p_2 }(a) \, \big (\langle a_{p_1,p_2} \rangle_0 \otimes \langle a_{ p_2, p_1} \rangle_0  - \langle a_{p_2, p_1} \rangle_0  \otimes \langle a_{p_1,p_2} \rangle_0  \big )  .$$ 
 \end{lemma}

\begin{proof}      We use the same $V'$ as in the proof of Lemma~\ref{rigovbvbq1a}  and apply Formula~\eqref{choricvxvxeee++veryryrsss} to compute $\nu_s(\langle a\rangle)$. The rest of the argument  uses the same ideas as  the proof of Lemma~\ref{rigovbvbq1a}.
  \end{proof}

\subsection{Star-fillings of~$X$}\label{Star-fillings} A \emph{star-filling} of~$X$ is  a finite family~$F$  of disjoint stars in~$X$ such that each component of   the set $\Sigma \setminus \cup_{s\in F} s$ is a disk   meeting  $\partial \Sigma \setminus \alpha $ at one or two open segments.

\begin{lemma}\label{rigoq1a}    If  the surface core~$\Sigma$ of~$X$ is compact  and each  component of~$\Sigma$  has a non-void boundary then~$X$ has a  star-filling.
 \end{lemma}

\begin{proof}  Cutting~$\Sigma$ along a finite set of disjoint proper segments $\{\beta_i\}_i$ we can get a disjoint union of closed disks $\{D_j\}_j$.
Pushing if necessary the endpoints of the segments  $\{\beta_i\}_i$ along $\partial \Sigma$ we can ensure that $\beta_i \cap \alpha= \emptyset$ for all~$i$. For each~$j$, the circle $\partial D_j$   is formed by several, say, $ n_j$ segments   in $\partial \Sigma \setminus \alpha$ and the same number of   segments which are either components of~$\alpha$ or copies of  some $\beta_i$'s.
  If $n_j \geq 2$, then we pick a star $s_j \subset D_j$ with center in $\Int (D_j)$ and with   leaves inside  the above-mentioned  $n_j$ 
  segments in $\partial \Sigma \setminus \alpha$.   Composing the inclusion $s_j \subset D_j$ with the natural embedding $D_j \hookrightarrow \Sigma$, we can view each $s_j$ as a star in~$\Sigma$.  Then the family $\{s_j \vert n_j \geq 2 \}_j$ is a star-filling of~$X$.
  \end{proof}
  
  \begin{lemma}\label{rigoq1-fillinga}   For any star-filling $F$ of~$X$, we have 
   \begin{equation}\label{choribneebn++C}
\bullet_{X}   =\sum_{s\in F}  \bullet_{s}  : H_1(X) \times H_1(X) \to \ZZ,   \end{equation}
 \begin{equation}\label{cddhoribnbnr++C}
 [-,-]_{X}   =   \sum_{s\in F}   [-,-]_{s}  : M\times M \to M,  \end{equation}
 \begin{equation}\label{chllsoribnbnm++C}
\nu_{X}   =   \sum_{s\in F}   \nu_{s}  : M \to M \otimes M.  
  \end{equation} \end{lemma}

\begin{proof} We  prove the first  equality, the other two are proven similarly. Without loss of generality we can assume that all stars in the family~$F$ lie in the reduced surface core $\Sigma' \subset \Sigma$. We choose  closed regular neighborhoods $\{V^s \supset s\}_{s\in F}$ so that they are pairwise disjoint and lie in  $\Sigma'$.  For each $s\in F$, let $\beta^s_1,..., \beta^s_{\vert s \vert}$ be the  segments in  $ \partial V^s $ separating $V^s$ from the rest of~$X$, cf. Section~\ref{Stars++}. Then the 1-manifold $\beta=\cup_{s\in F} \cup_{i=1}^{\vert s\vert} \beta^s_i$ satisfies the requirements needed to apply the transformation~C to~$X$. This transformation  produces  a quasi-surface $X^\beta$ having the same underlying topological space as~$X$. By  \eqref{choribnbn++C},   $
\bullet_{X}   = \bullet_{X^\beta}$.  Note that the surface core of $X^\beta$ is a disjoint union of the surfaces $\{V^s \supset s\}_{s\in F}$ and     $\overline{\Sigma \setminus \cup_{s\in F} V^s}$. By \eqref{choribnbn++}, the form  $\bullet_{X^\beta}$ is the sum of the 
forms $\bullet$ associated with the connected components of the surface core of   $X^\beta$. Each component   $V^s \supset s$ contributes the form  $\bullet_s$ to this sum. 
By the definition of a star-filling, all   components of the surface  $\overline {\Sigma \setminus \cup_{s\in F} V^s}$  are homeomorphic  to the quasi-surfaces described in  Section~\ref{A vanishing case}. Therefore the associated pairings $\bullet$ are equal to zero and 
$$\bullet_{X}   = \bullet_{X^\beta}  =\sum_{s\in F}  \bullet_{s}  : H_1(X) \times H_1(X) \to \ZZ.$$
  \end{proof}

\section{Proof of Theorems \ref{s-inter}--\ref{s-inter++cobr} and the case of  closed surfaces}\label{Proof of Theorems and their extension to closed surfaces}

\subsection{Proof of Theorems \ref{s-inter}--\ref{s-inter++cobr}}\label{Proof of Theorems and} Consider  the quasi-surface $X=X(\Gamma) $ with surface core~$\Gamma$,   empty singular core, and   $\alpha=\emptyset\subset \partial \Gamma$.  (Alternatively, one can use in this proof  a quasi-surface  with  the surface core~$\Gamma$, a 1-point singular core, and a set $\alpha\subset \partial \Gamma$ consisting of a single segment.) Any star~$s$ in~$\Gamma$ is a star in~$X$. 
Lemma~\ref{rigovbvbq1a}  implies that the skew-symmetric bilinear form $\bullet_s$ in $H_1(\Gamma)$ satisfies the conditions of Theorem~\ref{s-inter} and we  set  $\cdot_s=\bullet_s$. Note that $\bullet_X= 2 \, \cdot_X$ as is clear from  the remarks after the statement of Theorem~\ref{1aee}. Therefore  Theorem~\ref{s-inter2} is a direct consequence of Formula~\eqref{choribneebn++C}.
Similarly, the first claim  of Theorem~\ref{s-inter++} follows from Lemma~\ref{rigovbvbq1++++a},
and the second  claim  of Theorem~\ref{s-inter++} follows from Formula~\eqref{cddhoribnbnr++C}.
The claims of Theorem~\ref{s-inter++cobr} follow respectively  from Lemma~\ref{rigovbvbq1++++abbbn} and Formula~\eqref{chllsoribnbnm++C}.

\subsection{The case of closed surfaces}\label{The case of closed surfaces} Each closed    oriented surface~$\Phi$ gives rise to the homological  intersection form $\cdot_\Phi: H_1(\Phi) \times H_1(\Phi) \to \ZZ$, to 
  Goldman's bracket  $[-,-]_\Phi:M\times M\to M$, and  to the cobracket $ \nu_{\Phi}:M \to M \otimes M$. Here    $ M $  is  the free abelian group whose basis   is the set of free   homotopy classes of loops in~$\Phi$.  
 The definitions of $[-,-]_\Phi $ and $\nu_\Phi $ repeat   word for word  the definitions given in the introduction in the case of  surfaces with boundary. In this setting there seem to be no   analogues of the   maps $\cdot_s, [-,-]_s, \nu_s$ derived from stars. We compute  the maps $\cdot_\Phi, [-,-]_\Phi, \nu_{\Phi}$    in terms of so-called filling graphs  which we now define. 

By a \emph{bipartite graph} we mean a finite graph~$G$ provided with a partition of its  set of vertices   into two disjoint subsets $B(G)$ and $R(G)$ whose elements are called respectively the blue vertices and the red vertices so that every edge of~$G$  connects a blue vertex to a red vertex.    Note that a bipartite graph can not have edges connecting a vertex to itself. 
 
A \emph{filling graph} of~$\Phi$ is  a bipartite graph~$G$ embedded in~$\Phi$ so that (i) all components of $\Phi \setminus G$ are open disks and (ii) going along the boundary of each of these disks one traverses at most 4 edges and 4 vertices of~$G$. For example, any  triangulation~$\tau$ of~$\Phi$ determines a filling graph of~$\Phi$   formed by the vertices of~$\tau$ (declared to be \lq\lq blue''), the centers of the 2-simplices of~$\tau$  (declared to be \lq\lq red''), and the edges connecting the center of each 2-simplex of~$\tau$ to the   vertices of this 2-simplex. Exchanging the colors of the vertices, one derives from any filling graph of~$\Phi$ the dual filling graph. 

We show now how to calculate  the maps $\cdot_\Phi, [-,-]_\Phi$,   $ \nu_{\Phi}$ via a  filling graph $G\subset \Phi$. Note  that   the orientation of~$\Phi$ at any blue vertex  $v\in B(G)$ determines a cyclic order in the set  $\text{Edg}_v$ of the edges of~$G$ adjacent to~$v$. For   $e \in {\text {Edg}}_v$   we let $e^+ \in {\text {Edg}}_v$ be  the next   edge   with respect to this order. The set $ {\text {Edg}} (G)$  of all edges of~$G$ is a disjoint union of the sets $\{{\text {Edg}}_v\,  \vert \,  v\in B(G)\}$. Therefore the map $e\mapsto e^+$ is a permutation in  ${\text {Edg}} (G)$. 

We say that a loop   in~$\Phi$ is \emph{$G$-generic} if it   misses   the vertices of~$G$,   meets all  edges of~$G$  transversely, and never  traverses a point of~$G$ more than once. It is clear that any loop in~$\Phi$ can  be made $G$-generic by a small deformation.   We  orient each edge~$e$ of~$G$  from  its blue vertex  to its  red vertex.    For  a  $G$-generic loop~$a$ in~$\Phi$,     set $a\cap e =a(S^1) \cap  e $. The intersection sign of~$a$ and~$e$ at any point $p\in a\cap e$ is denoted by $\mu_p (a) $.   Set $a\cdot e =\sum_{p\in a \cap e}\mu_p(a)$.

\begin{theor}\label{rigoq1a}     For any $x,y \in H_1(\Phi)$ and any $G$-generic loops $a,b$ in~$\Phi$ representing $x,y$, we have 
  \begin{equation}\label{homolclosed} 2\, x \cdot_\Phi y = \sum_{e \in {\text {Edg}}(G)} \big (( a\cdot e) (b\cdot e^+) - (b\cdot e) ( a\cdot e^+) \big ). 
   \end{equation} 
 \end{theor}

\begin{proof}  Since the loops $a,b$ miss the  vertices of~$G$, each  red vertex $w\in R(G)$ has a closed disk neighborhood $D_w$ in $\Phi \setminus (a(S^1) \cup b(S^1))$.   We can assume that the disks $\{D_w\}_w$ are disjoint and  $D_w\cap G\subset \Phi$ is a star with center~$w$ for all $w$. 
Then $$\Gamma=\Phi \setminus \cup_{w\in R(G)} \Int (D_w)$$
is a compact connected subsurface of~$\Phi$  which we provide with orientation induced by that of~$\Phi$. For every blue vertex  $v\in B(G)$, the set
$$s_v=\Gamma \cap \cup_{e \in {\text {Edg}}_v} e$$
is a star in~$\Gamma$ with center~$v$. It is clear that the stars $\{ s_v \, \vert \, v\in R(G)\}$  are pairwise disjoint. The assumption that~$G$ is a filling graph of~$\Phi$ implies that this family of stars  is a star filling of~$\Gamma$. Since $a(S^1) \cup b(S^1) \subset \Gamma$, the loops $a,b$ represent certain homology classes $x', y' \in H_1(\Gamma)$. Theorems~\ref{s-inter} and~\ref{s-inter2} imply that
  \begin{equation}\label{homolclosedee} 2\, x' \cdot_\Gamma y' = \sum_{e \in {\text {Edg}}(G)} \big (( a\cdot e) (b\cdot e^+) - (b\cdot e) ( a\cdot e^+) \big ). \end{equation} 
On the other hand,  the inclusion homomorphism $H_1(\Gamma)\to H_1(\Phi)$ carries $x', y'$ respectively to $x,y$. The usual definition of the intersection forms $\cdot_\Gamma$ and $\cdot_\Phi$ implies that $x' \cdot_\Gamma y'= x \cdot_\Phi y$. Combining this   with~\eqref{homolclosedee}, we obtain~\eqref{homolclosed}. 
  \end{proof}

To state similar results for the bracket $ [-,-]_\Phi$  in the module~$M$, we need more notation. For a loop~$a$ in~$\Phi$, we let $\langle a \rangle\in  M$ be the free homotopy class of~$a$.
For a  point $p \in \Phi$ traversed by~$a$  only once, we let $a_p$ be the loop in~$\Phi$ starting  at~$p$ and going along~$a$ until the return to~$p$. We say that a pair of loops $a,b$ in~$\Phi$ is \emph{$G$-generic} if $a,b$ are $G$-generic and do not meet at   points of~$G$. 
Consider  a \emph{$G$-generic} pair of loops $a,b$ and consider  edges $e,e'$ of~$G$ sharing the same   blue vertex~$v$. For any points  $p\in a\cap e, q\in b\cap e'$, we  let $c=c_{p,q}$ be the path going  from~$p$ to~$v$ along~$e$ and then going from~$v$  to~$q$ along~$e'$.  We   write $a_pb_q$ for the loop $a_p c b_qc^{-1}$.

   \begin{theor}\label{s-inter++closed}   For any free homotopy classes $x,y$ of loops in~$\Phi$ and any $G$-generic pair of loops $a,b$ in~$\Phi$ representing $x,y$,  we have
$$
2\,   [x ,y]_\Phi $$
$$=  \sum_{e \in {\text {Edg}}(G)} \big ( \sum_{p\in a\cap e, q\in b\cap e^+} \mu_p(a)\,  \mu_q(b)  \langle  a_p b_q \rangle  -  \sum_{ p\in a\cap e^+, q\in b\cap e}
   \mu_p(a) \, \mu_q(b)  \langle      a_p b_q \rangle\big ) .  $$
\end{theor}

This theorem is deduced from Theorem~\ref{s-inter++}  using the   construction  introduced in     the proof of Theorem~\ref{rigoq1a}. 
  
  We finally   compute the cobracket  $ \nu_{\Phi}: M  \to M  \otimes M$. For a loop~$a$ in~$\Phi$, we set $\langle a \rangle_0=\langle a \rangle\in  M$   if~$a$ is non-contractible and $\langle a \rangle_0=0 \in  M$ if~$a$ is contractible. 
Consider a $G$-generic loop~$a$ in~$\Phi$  and edges $e,e'$ of~$G$ sharing the same   blue vertex~$v$. For   points  $p\in a\cap e, p'\in a\cap e'$, we     write $a_{p,p'}$ for the loop going from~$p$ to~$p'$  along~$a$, then going to~$v$ along~$e'$,  then going back to~$p$
along~$e$.

  \begin{theor}\label{s-inter++cobrclosed}  For any  free homotopy class~$x $ of loops in~$\Phi$ and any $G$-generic   loop~$a $ in~$\Phi$ representing~$x$,  we have   $$ 2\,  \nu_\Phi (x) $$
  $$= \sum_{e \in {\text {Edg}}(G)} \,  \sum_{p\in a\cap e, p'\in a\cap e^+}
  \mu_p(a)\,  \mu_{p'} (a)  \, \big (  \langle a_{p,p'} \rangle_0 \otimes \langle a_{p',p} \rangle_0  - \langle a_{p',p} \rangle_0  \otimes \langle a_{p,p'} \rangle_0 \big )   
.$$ 
\end{theor}

This theorem is deduced from Theorem~\ref{s-inter++cobr}  using the  construction in     the proof of Theorem~\ref{rigoq1a}. 

\subsection{Example} We check  Formula~\eqref{homolclosed} in a case where all computations can be done  explicitly.   Set $I=[0,1]$. 
Identifying the opposite  sides of the square $I^2$  via $(t,0)=(t, 1)$ and $(0,t)= (1,t)$ for all $t\in I$ we obtain a 2-torus~$T$. Let  
  $p: I^2\to T$ be  the projection. The  vertices $A_1=(0,0), A_2=(1,0)$, $A_3=(1,1)$,   $A_4=(0,1)$ of $I^2$ project to the same  point of~$T$   denoted~$A$.  Let $O=(1/2, 1/2)$ be the center of $I^2$.  Consider the graph $G\subset T$ with  two  vertices $p(O), A$  and four edges $\{e_i=p(OA_i)\}_{i=1}^4$ where $OA_i$ is the straight segment connecting~$O$ and $A_i$ in $I^2$. We let $p(O) $ be the blue vertex of~$G$ and let~$A$ be the red vertex of~$G$. This turns~$G$ into a filling graph of~$T$.   Consider the homology classes $x,y \in H_1(T)$ represented respectively by  the loops $$a:S^1 \to T, e^{2\pi i t} \mapsto p( (t, 1/3)) \quad {\text{and}} \quad b:S^1 \to T, e^{2\pi i t} \mapsto p ( ( 1/4, t)).$$ These loops meet in the  point $p((1/4, 1/3))$. We orient~$T$ so that $x \cdot_T y=+1$. 
Then
$$a \cdot e_1= a \cdot e_2= -1,  \quad  a \cdot e_3= a \cdot e_4=0$$
and 
$$b \cdot e_1= b \cdot e_4= 1,  \quad  b\cdot e_2= b \cdot e_3=0. $$
 Clearly,   $(e_i)^+=e_{i+1}$ for $i=1,2,3$  and $(e_4)^+=e_{1}$. Therefore
$$\sum_{e \in {\text {Edg}}(G)} \big (( a\cdot e) (b\cdot e^+) - (b\cdot e) ( a\cdot e^+) \big )=-(b \cdot e_1)  (a \cdot e_2) - (b \cdot e_4) (a \cdot e_1)=2  = 2\, x \cdot_T y$$
which checks Formula~\eqref{homolclosed} in this case.

\egm


\begin{thebibliography}{CJKLS}

\bibitem[AKKN1]{AKKN1}
A. Alekseev,  N. Kawazumi, Y. Kuno, F. Naef,
\emph{The Goldman-Turaev Lie bialgebra in genus zero and the Kashiwara-Vergne problem.}
Adv. Math. 326 (2018), 1–53.


\bibitem[AKKN2]{AKKN2}
A. Alekseev,  N. Kawazumi, Y. Kuno, F. Naef,
\emph{The Goldman-Turaev Lie bialgebra and the Kashiwara-Vergne problem in higher genera.}
arXiv:1804.09566.

 

 

 
 
 

\bibitem[Go1]{Go1}
W. M. Goldman,
\emph{ The symplectic nature of fundamental groups of surfaces.}
Adv. in Math. 54 (1984), no. 2, 200--225.

\bibitem[Go2]{Go2}
W. M. Goldman,
\emph{Invariant functions on Lie groups and Hamiltonian flows of surface group representations.}
Invent. Math. 85 (1986), no. 2, 263--302.

 
 

\bibitem[Ha1]{Hain}
R. Hain,
\emph{Hodge Theory of the Turaev Cobracket and the Kashiwara--Vergne Problem.}
arXiv:1807.09209.

 
\bibitem[Ha2]{Hain2}
R. Hain,
\emph{Johnson homomorphisms.}
  arXiv:1909.03914.

\bibitem[Ka]{Ka}  A. Kabiraj, \emph{Center of the Goldman Lie algebra.}
 Algebr. Geom. Topol. 16 (2016), no. 5, 2839–2849.

\bibitem[KK1]{KK}
N. Kawazumi, Y. Kuno,
\emph{The logarithms of Dehn twists.}
 Quantum Topol. 5 (2014), no. 3, 347-–423.

 
\bibitem[KK2]{KK2}
N. Kawazumi, Y. Kuno,
\emph{Intersection of curves on surfaces and their applications to mapping class groups.}
Ann. Inst. Fourier 65 (2015), no. 6, 2711–2762. 



\bibitem[LS]{LS} D. Li-Bland, P.  Severa,  
\emph{Moduli spaces for quilted surfaces and Poisson structures.}
Doc. Math. 20 (2015), 1071–1135. 

\bibitem[Ma]{Ma} G. Massuyeau, \emph{Formal descriptions of Turaev's loop operations.}  Quantum Topol. 9 (2018), no. 1, 39–117.
 
 
\bibitem[MT]{MT}  G. Massuyeau, V. Turaev,  
\emph{Quasi-Poisson structures on representation spaces of surfaces.}
Int. Math. Res. Not. (2014), no. 1, 1–64.



\bibitem[Tu1]{Tu1}
V.  Turaev,
\emph{Skein quantization of Poisson algebras of loops on surfaces.}
Ann. Sci. \'Ecole Norm. Sup. (4) 24 (1991), no. 6, 635--704.

\bibitem[Tu2]{Tu2}
V.  Turaev,
\emph{Topological constructions of tensor fields on moduli spaces.} arXiv:1901.02634.
 




                     \end{thebibliography}
    \end{document}